\documentclass[11pt]{amsart}

\usepackage{cite,amsmath,amssymb,amsthm,graphicx,mathptmx,tikz}
\usetikzlibrary{calc,decorations.markings,matrix,arrows}
\usepackage{subfig}

\providecommand{\tensor}{\otimes}
\providecommand{\affiliation}{\address}
\providecommand{\acknowledgments}{\section*{Acknowledgements}}

\newcommand{\C}{\mathbb{C}}
\newcommand{\Z}{\mathbb{Z}}
\newcommand{\R}{\mathbb{R}}
\renewcommand{\P}{\mathbb{P}}
\newcommand{\cC}{\mathcal{C}}
\newcommand{\cK}{\mathcal{K}}
\newcommand{\cO}{\mathcal{O}}

\newcommand{\CAT}{\mathrm{CAT}}
\newcommand{\Gr}{\mathrm{Gr}}
\newcommand{\Hom}{\mathrm{Hom}}
\newcommand{\Inv}{\mathrm{Inv}}
\newcommand{\Irr}{\mathrm{Irr}}
\newcommand{\PGL}{\mathrm{PGL}}
\newcommand{\Sch}{\mathrm{Sch}}
\newcommand{\SL}{\mathrm{SL}}
\renewcommand{\top}{\mathrm{top}}
\newcommand{\pt}{{\mathrm{pt}}}
\newcommand{\cf}{\C_c}

\newcommand{\fg}{\mathfrak{g}}
\newcommand{\Gv}{{G^\vee}}
\newcommand{\alphav}{\alpha^\vee}
\newcommand{\rhov}{\rho^\vee}
\newcommand{\vlambda}{{\vec{\lambda}}}
\newcommand{\vgamma}{{\vec{\gamma}}}
\newcommand{\vmu}{{\vec{\mu}}}
\newcommand{\vnu}{{\vec{\nu}}}
\newcommand{\vp}{\vec{p}}

\renewcommand{\tensor}{\otimes}
\newcommand{\longto}{\longrightarrow}

\newcommand{\defeq}{\stackrel{\mathrm{def}}{=}}
\newcommand{\eps}{\epsilon}
\renewcommand{\bar}{\overline}
\newcommand{\braket}[1]{{\langle #1 \rangle}}
\newcommand{\twistedprod}{\tilde{\times}}
\renewcommand{\min}{\textrm{min}}
\newcommand{\concat}{\sqcup}

\newcommand{\hconv}{\mathrm{\mathbf{hconv}}}
\newcommand{\econv}{\mathrm{\mathbf{econv}}}
\newcommand{\fsp}{\mathrm{\mathbf{fsp}}}
\newcommand{\perv}{\mathrm{\mathbf{perv}}}
\newcommand{\rep}{\mathrm{\mathbf{rep}}}
\newcommand{\spd}{\mathrm{\mathbf{spd}}}
\newcommand{\vect}{\mathrm{\mathbf{vect}}}
\newcommand{\svect}{\mathrm{\mathbf{svect}}}

\newtheorem{theorem}{Theorem}[section]
\newtheorem{proposition}[theorem]{Proposition}
\newtheorem{lemma}[theorem]{Lemma}
\newtheorem{question}[theorem]{Question}
\newtheorem{corollary}[theorem]{Corollary}
\newtheorem{conjecture}[theorem]{Conjecture}

\newenvironment{fullfigure}[2]
    {\begin{figure}[htb]\begin{center}\def\fullfiga{#1}\def\fullfigb{#2}}
    {\vspace{\baselineskip}\caption{\fullfigb.}\label{\fullfiga}
    \end{center}\end{figure}}
\newenvironment{fullfigure*}[2]
    {\begin{figure*}[htb]\begin{center}\def\fullfiga{#1}\def\fullfigb{#2}}
    {\vspace{\baselineskip}\caption{\fullfigb.}\label{\fullfiga}
    \end{center}\end{figure*}}

\newcommand{\eq}[2]{\begin{equation}\label{#1}#2\end{equation}}

\newcommand{\cor}[1]{Corollary~\ref{#1}}
\newcommand{\conj}[1]{Conjecture~\ref{#1}}
\newcommand{\fig}[1]{Figure~\ref{#1}}
\newcommand{\prop}[1]{Proposition~\ref{#1}}
\renewcommand{\sec}[1]{Section~\ref{#1}}
\newcommand{\thm}[1]{Theorem~\ref{#1}}
\newcommand{\lem}[1]{Lemma~\ref{#1}}

\newcommand{\ie}{\emph{i.e.}}
\newcommand{\Ie}{\emph{I.e.}}

\colorlet{lightgray}{black!15!white}
\colorlet{darkgreen}{green!75!black}
\colorlet{darkblue}{blue!75!black}
\colorlet{darkred}{red!75!black}

\tikzset{midto/.style={postaction={decorate,
    decoration={markings,mark=at position .5 with
    {\draw (-.035,-.07) -- (.035,0) -- (-.035,.07);}}}}}
\tikzset{midfrom/.style={postaction={decorate,
    decoration={markings,mark=at position .5 with
    {\draw (.035,-.07) -- (-.035,0) -- (.035,.07);}}}}}
\tikzset{web/.style={darkblue,semithick}}

\begin{document}

\title{Buildings, spiders, and geometric Satake}
\author{Bruce Fontaine}
\affiliation{University of Toronto}
\author{Joel Kamnitzer}
\thanks{The first and second authors were partly supported by NSERC}
\affiliation{University of Toronto}
\author{Greg Kuperberg}
\thanks{The third author was partly supported by
    NSF grants DMS-0606795 and CCF-1013079}
\affiliation{University of California, Davis}

\begin{abstract}
Let $G$ be a simple algebraic group.  Labelled trivalent graphs called
webs can be used to product invariants in tensor products of minuscule
representations.  For each web, we construct a configuration space of points
in the affine Grassmannian.  Via the geometric Satake correspondence,
we relate these configuration spaces to the invariant vectors coming
from webs.  In the case $G =\SL(3)$, non-elliptic webs yield a basis for
the invariant spaces.  The non-elliptic condition, which is equivalent to
the condition that the dual diskoid of the web is $\CAT(0)$, is explained
by the fact that affine buildings are $\CAT(0)$.
\end{abstract}
\maketitle

\section{Introduction}

\subsection{Spiders}

Let $G$ be a simple, simply-connected complex algebraic group. In previous
work \cite{Kuperberg:spiders}, the third author defined a pivotal
tensor category with generators and relations called a ``spider", for
$G$ of rank 2.  (The term ``spider" was originally intended to mean
any pivotal category, but in common usage only these categories are
called spiders.)  The Karoubi envelope of this category is equivalent
to the category $\rep^u(G)$ of finite-dimensional representations of $G$
with a modified pivotal structure.  Actually, the spider comes with a
parameter $q$ making it equivalent to the quantum deformation $\rep^u_q(G)$.
These results in rank 2 are analogous to the influential result of Kauffman
\cite{Kauffman:spinknot} and Penrose \cite{Penrose:negative} that the Karoubi
envelope of the Temperley-Lieb category (the category of planar matchings)
is equivalent to $\rep^u_q(\SL(2))$.  The Temperley-Lieb category can thus
be called the $\SL(2)$ spider.  Conjectural generalizations of spiders
were proposed for $\SL(4)$ by Kim \cite{Kim:thesis} and for $\SL(n)$
by Morrison \cite{Morrison:diagram}.

In this article, for any $G$ as above, we will define the free spider for
$G$ generated by the minuscule representations of $G$.  A morphism in the
free spider is given by a (linear combination) of labelled trivalent graphs
called webs.  For each web $w$ with boundary edges labelled $\vlambda$,
there is an invariant vector
$$\Psi(w) \in \Inv(V(\vlambda)) = \Inv_G(V(\lambda_1) \tensor V(\lambda_2)
    \tensor \cdots \tensor V(\lambda_n)).$$

If $G$ has rank 1 or 2, then the vectors $\Psi(w)$ coming from non-elliptic
webs $w$ (those whose faces have non-positive combinatorial curvature)
form a basis of each invariant space $\Inv(V(\vlambda))$ of $G$, called
a web basis.  The web basis for $\SL(2)$ is well-known as the basis
of planar matchings and it is known to be the same as Lusztig's dual
canonical basis \cite{FK:canonical}.  On the other hand, the $\SL(3)$
web bases are eventually not dual canonical \cite{Kuperberg:notdual},
even though many basis vectors are dual canonical.

\subsection{Affine Grassmannians}

The goal of this article is to introduce a new geometric interpretation
of webs and spiders using the geometry of affine Grassmannians.

Let $\cO = \C[[t]]$ and $\cK = \C((t))$.  In order to study the
representation theory of $G$, we will consider the affine Grassmannian of
its Langlands dual group
$$\Gr = \Gr(\Gv) = \Gv(\cK)/\Gv(\cO).$$
The geometric Satake correspondence of Lusztig \cite{Lusztig:qanalog},
Ginzburg \cite{Ginzburg:loop}, and Mirkovi\'c-Vilonen \cite{MV:geometric}
will be our main tool in this article.

\begin{theorem}  The category of equivariant perverse sheaves on the affine
Grassmannian $\Gr$ is equivalent as a symmetric and pivotal tensor category
to the tensor category $\rep^u(G)$ of representations of $G$ with a modified
pivotal and symmetric structure.
\label{th:satake} \end{theorem}

As a consequence of this theorem, every invariant space $\Inv(V(\vlambda))$
for every $G$ can be constructed from the geometry of $\Gr$.  Given a
vector $\vlambda$ of dominant weights of $G$, there is a convolution morphism
$$m_\vlambda:\bar{\Gr(\vlambda)} = \bar{\Gr(\lambda_1)} \twistedprod
    \bar{\Gr(\lambda_2)} \twistedprod \cdots \twistedprod \bar{\Gr(\lambda_n)}
    \longto \Gr,$$
where each $\Gr(\lambda)$ is a sphere of radius $\lambda$ (in the sense of
weight-valued distances \cite{KLM:generalized}) in $\Gr$.  The fibre
$F(\vlambda) = m_\vlambda^{-1}(t^0)$ is a projective variety that we call
the Satake fibre.  In particular, we will use the following corollary of
the geometric Satake correspondence.

\begin{theorem} Every invariant space in $\rep^u(G)$ is canonically
isomorphic to the top homology of the corresponding geometric Satake fibre
with complex coefficients:
$$\Phi:\Inv(V(\vlambda)) \cong H_\top(F(\vlambda),\C).$$
Each top-dimensional component $Z \subseteq F(\vlambda)$ thus yields
a vector $[Z] \in \Inv(V(\vlambda))$.  These vectors form a basis, the
Satake basis.
\label{th:satbasis} \end{theorem}

A goal of this article is to understand how the invariant vectors coming
from webs expand in this basis.  (Throughout, we will assume complex
coefficients for homology and cohomology.)

\subsection{Diskoids}

The orbits of $G(\cK)$ on the affine Grassmannian defines a notion
of distance on $\Gr$ with values in the set of dominant weights for
$G$.  Thus, we can interpret $F(\vlambda)$ as the (contractive, based)
configuration space in $\Gr$ of an abstract polygon $P(\vlambda)$ whose
side lengths are
$$\vlambda = (\lambda_1,\lambda_2,\ldots,\lambda_n).$$
One of our ideas is to generalize this type of configuration space
from polygons to diskoids.  For us, a diskoid $D$ is a contractible
piecewise linear region in the plane; in many cases it is a disk.  (See
\sec{s:config}.)  If $D$ is tiled by polygons and its edges are labelled
by dominant weights, then its vertices are a weight-valued metric space.
We will define a (based) configuration space $Q(D)$ which consists of maps
from the vertices of $D$ to $\Gr$ that preserves the lengths of edges
of $D$.  We will also define a special subset $Q_g(D)$ that consists of
maps that preserve all distances (globally isometric embeddings).

Assume that $\vlambda$ is a vector of minuscule highest weights.  If $w$
is a web with boundary $\vlambda$, then it has a dual diskoid $D = D(w)$
(or possibly a diskoid with bubbles).  The boundary of this diskoid is a
polygon $P(\vlambda)$ and so we get a map of configuration spaces $\pi:Q(D)
\to F(\vlambda)$.  Our first main result is that we can recover the vector
$\Psi(w)$ using this geometry.

\begin{theorem} There exists a homology class $c(w) \in H_*(Q(D))$ such
that $\pi_*(c(w)) \in H_{\top}(F(\vlambda))$ corresponds to $\Psi(w)$
under the isomorphism from \thm{th:satbasis}.
\label{th:homclass} \end{theorem}

We prove this theorem as an application of the geometric Satake
correspondence.  In many cases, the class $c(w)$ is the fundamental class
of $Q(D)$, so that the coefficients of $\pi_*(c(w))$ (and hence $\Psi(w)$)
in the Satake basis are just the degrees of the map $\pi$ over the components
of $F(\vlambda)$.

\subsection{Buildings}

The affine Grassmannian $\Gr$ embeds isometrically into the affine
building $\Delta = \Delta(\Gv)$.  We can use this perspective to gain
greater insight into the variety $Q(D)$.

If $G = \SL(2)$, then a basis web is a planar matching (or cup diagram)
and its dual diskoid $D$ is a finite tree.  The affine Grassmannian $\Gr$
is the set of vertices of the affine building $\Delta$, which is an infinite
tree with infinite valence.  The configuration space $Q(D)$ is the space
of colored, based simplicial maps $f:D \to \Delta$; see \fig{f:a1example}.
It is known that
$$Q(D) = \P^1 \twistedprod \P^1 \twistedprod \cdots \twistedprod \P^1$$
is a twisted product of $\P^1$'s, and that these twisted products are
the components of the Satake fibre $F(\vlambda)$.  Moreover, $Q_g(D)$ is
the open dense subvariety of points in $Q(D)$ which are contained in no
other component of $F(\vlambda)$.  \fig{f:a1example} is an illustration
of the construction.

\begin{fullfigure*}{f:a1example}{From a non-elliptic $A_1$ web, to a tree,
    to part of an affine $A_1$ building}
\begin{tikzpicture}[baseline]
\draw (0,0) circle (1);
\foreach \th in {0,3,5} {
    \draw[web] (\th*45+67.5:1) .. controls (\th*45+67.5:.6)
        and (\th*45+112.5:.6) .. (\th*45+112.5:1); };
    \draw[web] (22.5:1) .. controls (22.5:.5)
        and (157.5:.5) .. (157.5:1);
\end{tikzpicture}
{\large $\longto$}
\begin{tikzpicture}[baseline]
\coordinate (a) at (0,0); \coordinate (a0) at (90:.75);
\coordinate (a1) at (210:.75); \coordinate (a2) at (330:.75);
\coordinate (a00) at (90:1.5);
\draw (a00) -- (a) -- (a1); \draw (a) -- (a2);
\fill (a) circle (.075); \fill (a00) circle (.075);
\draw[fill=white] (a0) circle (.075);
\draw[fill=white] (a1) circle (.075);
\draw[fill=white] (a2) circle (.075);
\end{tikzpicture}
{\large $\longto$}
\begin{tikzpicture}[baseline]
\coordinate (a) at (0,0);
\foreach \x in {0,1,2} {
    \coordinate (a\x) at (120*\x+90:.75);
    \foreach \y in {0,1} {
        \coordinate (a\x\y) at (120*\x+60*\y+60:1.5);
        \foreach \z in {0,1} {
            \coordinate (a\x\y\z) at (120*\x+60*\y+30*\z+45:2.25);
            \foreach \w in {0,1} {
                \coordinate (a\x\y\z\w) at (120*\x+60*\y+30*\z+10*\w+40:2.5);
}; }; }; };
\foreach \color/\r in {darkblue/.25,white/.225}
{
    \path[fill=\color] (a) circle (\r);
    \path[fill=\color] (a0) circle (\r);
    \path[fill=\color] (a1) circle (\r);
    \path[fill=\color] (a2) circle (\r);
    \path[fill=\color] (a00) circle (\r);
    \path[draw=\color,line width=\r*2cm] (a0) -- (a00);
    \path[draw=\color,line width=\r*2cm] (a) -- (a0);
    \path[draw=\color,line width=\r*2cm] (a) -- (a1);
    \path[draw=\color,line width=\r*2cm] (a) -- (a2);
}
\fill[white] (a00) circle (.225);
\fill (a) circle (.075);
\foreach \x in {0,1,2} {
    \foreach \r in {2.75,3,3.25} { \fill (120*\x+90:\r) circle (.035); };
    \draw (a) -- (a\x);
    \foreach \y in {0,1} {
        \draw (a\x) -- (a\x\y);
        \fill (a\x\y) circle (.075);
        \foreach \z in {0,1} {
            \draw (a\x\y) -- (a\x\y\z);
            \foreach \w in {0,1} {
                \draw (a\x\y\z) -- (a\x\y\z\w); };
            \draw[fill=white] (a\x\y\z) circle (.075); }; };
    \draw[fill=white] (a\x) circle (.075); };
\end{tikzpicture}
\end{fullfigure*}

\begin{fullfigure*}{f:a2example}{From a non-elliptic $A_2$ web, to a $\CAT(0)$
    diskoid, to part of an affine $A_2$ building}
\begin{tikzpicture}[baseline]
\foreach \t in {0,90,180,270} {
    \draw[web,midto] (\t+45:.75) -- (\t:.75);
    \draw[web,midto] (\t-45:.75) -- (\t:.75);
    \draw[web,midto] (\t:1.3) -- (\t:.75);
    \draw[web,midto] (\t+45:.75) -- (\t+45:1.3); };
\draw[web,midto] (30:1.6) -- (45:1.3);
\draw[web,midto] (60:1.6) -- (45:1.3);
\draw[black] plot[smooth cycle] coordinates{(0:1.3) (30:1.6) (45:1.7)
    (60:1.6) (90:1.3) (112.5:1.3) (135:1.3) (157.5:1.3) (180:1.3) (202.5:1.3)
    (225:1.3) (247.5:1.3) (270:1.3) (292.5:1.3) (315:1.3) (337.5:1.3)};
\end{tikzpicture}
{\large $\longto$}
\begin{tikzpicture}[baseline]
\draw[fill=lightgray] (22.5:1) -- (45:1.6) -- (67.5:1);
\foreach \t in {0,45,...,315} {
    \fill[lightgray] (0,0) -- (\t+22.5:1) -- (\t-22.5:1); };
\foreach \t in {0,45,...,315} {
    \draw (0,0) -- (\t+22.5:1) -- (\t-22.5:1); };
\fill[blue] (45:1.6) circle (.07);
\fill[blue] (0,0) circle (.07);
\foreach \t in {0,90,180,270} {
    \fill[red] (\t-22.5:1) circle (.07);
    \fill[darkgreen] (\t+22.5:1) circle (.07); };
\end{tikzpicture}
{\large $\longto$}
\begin{tikzpicture}[baseline]
\coordinate (p31) at (1.861,0.589); \coordinate (p27) at (-0.433,0.422);
\coordinate (p05) at (1.926,0.997); \coordinate (p12) at (1.974,-1.439);
\coordinate (p09) at (-1.926,1.514); \coordinate (p28) at (0.816,-0.892);
\coordinate (p17) at (1.315,1.641); \coordinate (p21) at (-0.448,-2.385);
\coordinate (p10) at (1.131,-0.561); \coordinate (p15) at (-1.974,-1.073);
\coordinate (p25) at (0.448,-1.274); \coordinate (p08) at (-1.592,0.870);
\coordinate (p01) at (0.000,0.000); \coordinate (p33) at (-1.861,-0.756);
\coordinate (p14) at (-2.309,-0.429); \coordinate (p19) at (0.000,-1.450);
\coordinate (p20) at (-1.362,-1.716); \coordinate (p18) at (-0.269,1.254);
\coordinate (p04) at (1.592,-0.703); \coordinate (p32) at (-0.612,-2.070);
\coordinate (p23) at (0.269,2.406); \coordinate (p13) at (-1.131,-1.113);
\coordinate (p29) at (0.414,0.653); \coordinate (p06) at (-0.653,1.315);
\coordinate (p35) at (-0.816,1.059); \coordinate (p22) at (-1.315,2.019);
\coordinate (p30) at (0.612,-0.442); \coordinate (p16) at (0.000,1.450);
\coordinate (p02) at (0.653,0.359); \coordinate (p11) at (2.309,0.261);
\coordinate (p24) at (1.362,-1.944); \coordinate (p26) at (-0.414,-0.653);
\coordinate (p07) at (-1.545,0.175); \coordinate (p03) at (1.545,-0.175);
\coordinate (p34) at (0.433,2.090);
\fill[lightgray,nearly opaque] (p27) -- (p02) -- (p26);
\draw[semithick] (p27) -- (p02) -- (p26) -- cycle;
\fill[blue] (p27) circle (.07);
\fill[lightgray,nearly opaque] (p28) -- (p02) -- (p26);
\draw[semithick] (p28) -- (p02) -- (p26) -- cycle;
\fill[blue] (p28) circle (.07);
\fill[lightgray,nearly opaque] (p18) -- (p02) -- (p16);
\draw[semithick] (p18) -- (p02) -- (p16) -- cycle;
\fill[blue] (p18) circle (.07);
\fill[lightgray,nearly opaque] (p04) -- (p02) -- (p03);
\draw[semithick] (p04) -- (p02) -- (p03) -- cycle;
\fill[blue] (p04) circle (.07);
\fill[lightgray,nearly opaque] (p17) -- (p02) -- (p16);
\draw[semithick] (p17) -- (p02) -- (p16) -- cycle;
\fill[blue] (p17) circle (.07);
\fill[lightgray,nearly opaque] (p33) -- (p13) -- (p26);
\draw[semithick] (p33) -- (p13) -- (p26) -- cycle;
\fill[blue] (p33) circle (.07);
\fill[lightgray,nearly opaque] (p05) -- (p02) -- (p03);
\draw[semithick] (p05) -- (p02) -- (p03) -- cycle;
\fill[blue] (p05) circle (.07);
\fill[lightgray,nearly opaque] (p01) -- (p02) -- (p26);
\draw[semithick] (p01) -- (p02) -- (p26) -- cycle;
\fill[lightgray,nearly opaque] (p32) -- (p13) -- (p26);
\draw[semithick] (p32) -- (p13) -- (p26) -- cycle;
\fill[blue] (p32) circle (.07);
\fill[lightgray,nearly opaque] (p01) -- (p02) -- (p16);
\draw[semithick] (p01) -- (p02) -- (p16) -- cycle;
\fill[lightgray,nearly opaque] (p01) -- (p02) -- (p03);
\draw[semithick] (p01) -- (p02) -- (p03) -- cycle;
\fill[darkgreen] (p02) circle (.07);
\fill[lightgray,nearly opaque] (p01) -- (p13) -- (p26);
\draw[semithick] (p01) -- (p13) -- (p26) -- cycle;
\fill[red] (p26) circle (.07);
\fill[lightgray,nearly opaque] (p14) -- (p13) -- (p07);
\draw[semithick] (p14) -- (p13) -- (p07) -- cycle;
\fill[blue] (p14) circle (.07);
\fill[lightgray,nearly opaque] (p22) -- (p06) -- (p16);
\draw[semithick] (p22) -- (p06) -- (p16) -- cycle;
\fill[blue] (p22) circle (.07);
\fill[lightgray,nearly opaque] (p01) -- (p06) -- (p16);
\draw[semithick] (p01) -- (p06) -- (p16) -- cycle;
\fill[lightgray,nearly opaque] (p21) -- (p13) -- (p19);
\draw[semithick] (p21) -- (p13) -- (p19) -- cycle;
\fill[blue] (p21) circle (.07);
\fill[lightgray,nearly opaque] (p01) -- (p13) -- (p07);
\draw[semithick] (p01) -- (p13) -- (p07) -- cycle;
\fill[lightgray,nearly opaque] (p23) -- (p06) -- (p16);
\draw[semithick] (p23) -- (p06) -- (p16) -- cycle;
\fill[red] (p16) circle (.07); \fill[blue] (p23) circle (.07);
\fill[lightgray,nearly opaque] (p01) -- (p13) -- (p19);
\draw[semithick] (p01) -- (p13) -- (p19) -- cycle;
\fill[lightgray,nearly opaque] (p01) -- (p10) -- (p03);
\draw[semithick] (p01) -- (p10) -- (p03) -- cycle;
\fill[lightgray,nearly opaque] (p12) -- (p10) -- (p03);
\draw[semithick] (p12) -- (p10) -- (p03) -- cycle;
\fill[blue] (p12) circle (.07);
\fill[lightgray,nearly opaque] (p09) -- (p06) -- (p07);
\draw[semithick] (p09) -- (p06) -- (p07) -- cycle;
\fill[blue] (p09) circle (.07);
\fill[lightgray,nearly opaque] (p01) -- (p06) -- (p07);
\draw[semithick] (p01) -- (p06) -- (p07) -- cycle;
\fill[lightgray,nearly opaque] (p15) -- (p13) -- (p07);
\draw[semithick] (p15) -- (p13) -- (p07) -- cycle;
\fill[blue] (p15) circle (.07);
\fill[lightgray,nearly opaque] (p11) -- (p10) -- (p03);
\draw[semithick] (p11) -- (p10) -- (p03) -- cycle;
\fill[red] (p03) circle (.07); \fill[blue] (p11) circle (.07);
\fill[lightgray,nearly opaque] (p20) -- (p13) -- (p19);
\draw[semithick] (p20) -- (p13) -- (p19) -- cycle;
\fill[darkgreen] (p13) circle (.07); \fill[blue] (p20) circle (.07);
\fill[lightgray,nearly opaque] (p01) -- (p10) -- (p19);
\draw[semithick] (p01) -- (p10) -- (p19) -- cycle;
\fill[lightgray,nearly opaque] (p01) -- (p06) -- (p29);
\draw[semithick] (p01) -- (p06) -- (p29) -- cycle;
\fill[lightgray,nearly opaque] (p24) -- (p10) -- (p19);
\draw[semithick] (p24) -- (p10) -- (p19) -- cycle;
\fill[blue] (p24) circle (.07);
\fill[lightgray,nearly opaque] (p01) -- (p10) -- (p29);
\draw[semithick] (p01) -- (p10) -- (p29) -- cycle;
\fill[blue] (p01) circle (.07);
\fill[lightgray,nearly opaque] (p08) -- (p06) -- (p07);
\draw[semithick] (p08) -- (p06) -- (p07) -- cycle;
\fill[red] (p07) circle (.07); \fill[blue] (p08) circle (.07);
\fill[lightgray,nearly opaque] (p34) -- (p06) -- (p29);
\draw[semithick] (p34) -- (p06) -- (p29) -- cycle;
\fill[blue] (p34) circle (.07);
\fill[lightgray,nearly opaque] (p25) -- (p10) -- (p19);
\draw[semithick] (p25) -- (p10) -- (p19) -- cycle;
\fill[red] (p19) circle (.07); \fill[blue] (p25) circle (.07);
\fill[lightgray,nearly opaque] (p31) -- (p10) -- (p29);
\draw[semithick] (p31) -- (p10) -- (p29) -- cycle;
\fill[blue] (p31) circle (.07);
\fill[lightgray,nearly opaque] (p35) -- (p06) -- (p29);
\draw[semithick] (p35) -- (p06) -- (p29) -- cycle;
\fill[darkgreen] (p06) circle (.07); \fill[blue] (p35) circle (.07);
\fill[lightgray,nearly opaque] (p30) -- (p10) -- (p29);
\draw[semithick] (p30) -- (p10) -- (p29) -- cycle;
\fill[darkgreen] (p10) circle (.07); \fill[red] (p29) circle (.07);
\fill[blue] (p30) circle (.07);
\foreach \th in {90,210,330} {
    \foreach \r in {2.8,3,3.2} { \fill (\th:\r) circle (.03); }; };
\end{tikzpicture}
\end{fullfigure*}

Our other main results are a generalization of this fact to $G = \SL(3)$.
In this case, $\Gr$ is again the vertex set of $\Delta$.  If $w$ is a
non-elliptic web with boundary $\vlambda$, then $Q(D(w))$ is again the space
of colored, based simplicial maps $f:D \to \Delta$, as in \fig{f:a2example}.
Then:

\begin{theorem} Let $G = \SL(3) = A_2$ and let $w$ be a non-elliptic web
with minuscule boundary $\vlambda$ and dual diskoid $D$.  Then the global
isometry configuration space $Q_g(D)$ is mapped isomorphically by $\pi$
to a dense subset of a component of the Satake fibre $F(\vlambda)$.
This inclusion yields a bijection between non-elliptic webs and the
components of $F(\vlambda)$.
\label{th:main} \end{theorem}

Our construction can be viewed as an explanation of why basis webs
are non-elliptic.  A web is non-elliptic if and only if its diskoid is
$\CAT(0)$, essentially by definition.  It is well-known that every affine
buildings is a $\CAT(0)$ space \cite{BT:local1}.  Moreover, every convex
subset of a $\CAT(0)$ space, such as a diskoid which is isometrically
embedded in a building, is necessarily $\CAT(0)$.  We will also show
that the image of each diskoid embedding $f:D \to \Delta$ in $Q_g(D)$
has a least area property.  Likewise, the elliptic relations of the $A_2$
spider can be viewed as area-decreasing transformations.

Meanwhile, if $w$ is non-elliptic, then $Q(D)$ is sometimes the closure
of $Q_g(D)$ and hence maps to a single component of $F(\vlambda)$.
Eventually, $Q(D)$ has other components and maps to more than one
component of $F(\vlambda)$.  These other components seem related to the
phenomenon that web bases are not dual canonical.  However, we can get
an upper triangularity result as follows.  In \sec{s:cyclic},
we will define a partial order $\le_S$ on the set of non-elliptic webs
using pairwise distances between boundary vertices of their dual diskoids.

\begin{theorem} The change of basis in $\Inv(V(\vlambda))$ from non-elliptic
webs to the Satake basis is unitriangular, relative to the partial order
$\le_S$.
\label{th:unitriang} \end{theorem}

We have learned from Sergei Ivanov \cite{Ivanov:personal} that the partial
order in \thm{th:unitriang} refines the partial order on webs given
by the number of vertices.

Also, in \sec{s:notsatake}, we will show that the web basis, the Satake
basis, and the dual canonical basis for $\SL(3)$ are all eventually
different.

Finally, in \sec{s:euler}, we will propose a different formulation of the
geometric Satake correspondence based on convolution of constructible
functions rather than convolution of homology classes.  (In
\thm{th:convperv}, we reinterpret geometric Satake in terms of convolution
in homology).  We will prove this conjecture in the case of a tensor
product of minuscule representations of $\SL(3)$.

\subsection{Satake fibres and Springer fibres}

When $G = \SL(m)$ and $\vlambda = (\omega_1, \dots, \omega_1)$ is an $n =
mk$ tuple consisting of $\omega_1$ (the highest weight of the standard
representation), then $F(\vlambda)$ is isomorphic to the $(k, k, \dots,
k)$ Springer fibre.  In other words, $F(\vlambda)$ is the variety of
flags in $\C^n$ invariant under a nilpotent endomorphism with $m$
Jordan blocks all of size $k \times k$.  We have already mentioned the
well-known description of the components of the Springer or Satake fibre
in terms of planar matchings when $m=2$.  This Springer fibre formalism
and this description of it have been used as a model of Khovanov homology
\cite{Khovanov:springer,Stroppel:springer}.  One motivation for the present
work is to generalize this result to case $m = 3$ and obtain a description
of the components of the Springer or Satake fibre using non-elliptic webs.
\thm{th:main} accomplishes this task.  (See also the end of the introduction
of \cite{Tymoczko:bijection}.)

\begin{fullfigure*}{f:joke}{Spiders and buildings}
\frame{\includegraphics[height=3in]{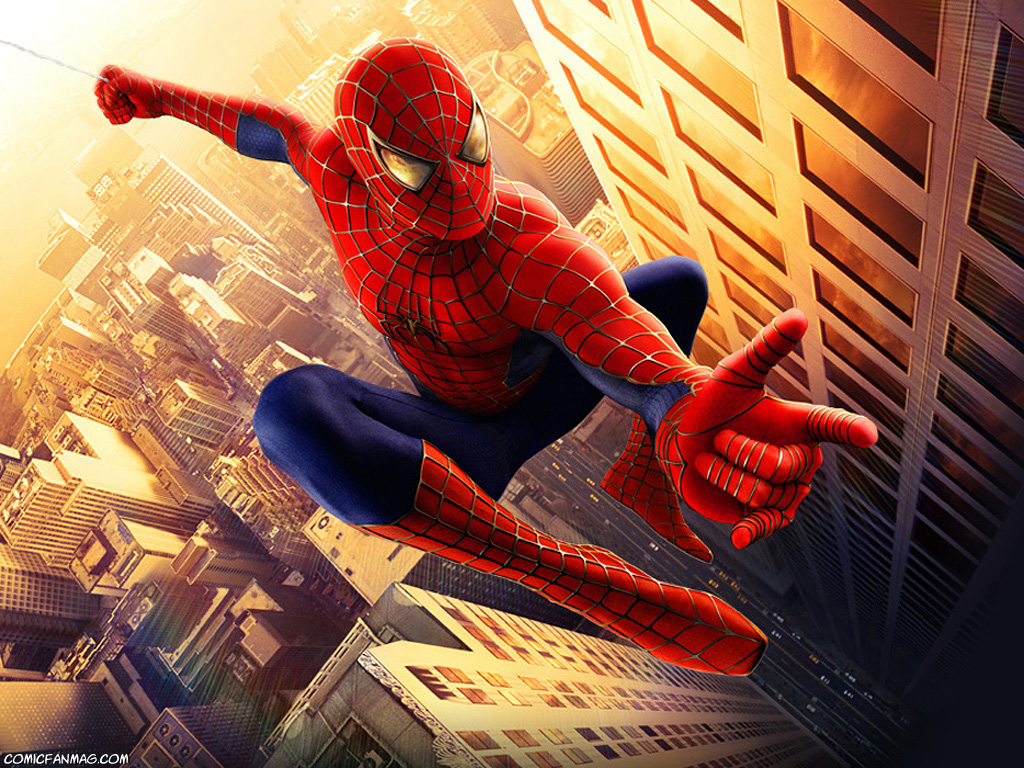}}
\end{fullfigure*}

\acknowledgments

The authors would like to thank Dave Anderson, Charles Frohman, Dennis
Gaitsgory, Andr\'e Henriques, Misha Kapovich, Anthony Licata, John Millson,
Scott Morrison, Hiraku Nakajima, Alistair Savage, Petra Schwer, and Juliana
Tymoczko for useful discussions.

Another motivation for our work is shown in \fig{f:joke}.

\section{Spiders}
\label{s:spiders}

\subsection{Pivotal and symmetric categories}
\label{s:pivotal}

The definitions used in this section are nicely summarized in a survey by
Selinger \cite{Selinger:survey}; they are originally due to Freyd-Yetter
\cite{FY:braided} and Joyal-Street \cite{JS:braided}.

A \emph{pivotal category} $\cC$ is a (strict) monoidal tensor category
such that each object $A$ has a two-sided dual object $A^*$.  This means
that there is a contravariant functor $F(A) = A^*$ from $\cC$ to itself
which is also an order-reversing tensor functor, \ie,
$$(A \tensor B)^* = B^* \tensor A^*,$$
and which has these extra properties:
For each object $A$, there are ``cup" and ``cap" morphisms
$$b_A:I \longto A^* \tensor A \qquad d_A:A \tensor A^* \longto I$$
where $I$ denotes the unit object, such that
$$(1_A \tensor d_A)(b_A \tensor 1_A) = 1_A \qquad
    (d_{A_*} \tensor 1_{A_*})(1_{A_*} \tensor b_{A_*}) = 1_{A^*}.$$
In addition, $*$ is an anti-involution of the category $\cC$.  (We assume
that $*$ is a strict involution of $\cC$ that reverses both tensor products
and compositions of morphisms.)  The axiom can be graphically summarized
as follows:
\eq{e:plumbing}{
\begin{tikzpicture}[scale=.75,baseline=-.5ex,allow upside down]
\draw[midto,web] (-1,-1) -- (-1,0);
\draw[web] (-1,0) arc (180:0:.5) arc (180:360:.5);
\draw[midto,web] (1,0) -- (1,1);
\node[anchor=south] at (-.5,.5) {$d_A$};
\node[anchor=north] at (.5,-.5) {$b_A$};
\node[anchor=east] at (-1,-.5) {$A$};
\draw[dashed] (-1.5,0) -- (1.5,0);
\end{tikzpicture} \;=\;
\begin{tikzpicture}[scale=.75,baseline=-.5ex]
\draw[web,midto] (1,-1) -- (1,0);
\draw[web] (1,0) arc (0:180:.5) arc (360:180:.5);
\draw[web,midto] (-1,0) -- (-1,1);
\node[anchor=south] at (.5,.5) {$d_{A^*}$};
\node[anchor=north] at (-.5,-.5) {$b_{A^*}$};
\node[anchor=west] at (1,-.5) {$A$};
\draw[dashed] (-1.5,0) -- (1.5,0);
\end{tikzpicture} \;=\;
\begin{tikzpicture}[scale=.75,baseline=-.5ex]
\draw[web,midto] (0,-1) -- (0,1);
\node[anchor=east] at (0,0) {$A$};
\end{tikzpicture}\;.}
A \emph{pivotal functor} is a tensor functor that preserves the above
structure.

Every object $A$ in a monoidal category has an \emph{invariant space}
$$\Inv(A) \defeq \Hom(I,A).$$
If the category is pivotal, then each invariant space has two other
important properties.  First, every space of morphisms is an invariant
space by the relation
$$\Hom(A,B) \cong \Inv(A^* \tensor B).$$
Second, there is a cyclic action on the invariant spaces in tensor products
$$R:\Inv(A \tensor B) \stackrel{\cong}{\longto} \Inv(B \tensor A),$$
which we call a \emph{rotation map}.  It extends to a rotation
of $n$ tensor factors:
$$R:\Inv(A_1 \tensor \cdots \tensor A_n) \stackrel{\cong}{\longto}
    \Inv(A_2 \tensor \cdots \tensor A_n \tensor A_1).$$

Another way to describe a pivotal category, already suggested in equation
\eqref{e:plumbing}, is that it has the structure to evaluate a planar
graph $w$ drawn in a disk, if the edges of $w$ are oriented and labelled
by objects and the vertices are labelled by invariants.  (The literature
uses the words ``labelled" and ``colored" interchangeably here; Selinger
\cite{Selinger:survey} calls an allowed set of colors a ``signature".)
The value of such a graph $w$ is another invariant, taking values in the
invariant space of the boundary of $w$.  The graph is considered up to
isotopy rel boundary, and an edge labelled by $A$ is equivalent to the
opposite edge labelled by $A^*$.  It is possible to write axioms for a
pivotal category using invariants and planar graphs rather than morphisms.
From this viewpoint, a word in a pivotal category is such a graph and it
can be called a \emph{web}.

A web is a special case of a \emph{ribbon graph} \cite{RT:ribbon}, the
difference being that a ribbon graph can also have crossings.
A \emph{braided category} is a monoidal category with crossing
isomorphisms
$$c_{A,B}:A \tensor B \to B \tensor A$$
that satisfy suitable axioms so that, among other things, a braid group acts
on the invariant space of a tensor product.  If the crossing isomorphisms
are involutions, then the braid group action descends to a symmetric group
action and the monoidal category is called \emph{symmetric}.  If a category
is both symmetric and pivotal, then there is an important compatibility
condition that together makes it a \emph{compact closed category}.
We require that the two involutions on $\Inv(A \tensor A)$, one coming from
the pivotal structure and the other from the symmetric structure, agree.
Equivalently, we require that
$$c_{A^*,A}(b_A) = b_{A^*}.$$
In a compact closed symmetric category, abstract graphs $w$ can be evaluated
whether or not they are planar.

Two other intermediate types of categories between pivotal and compact closed
are \emph{ribbon categories} and \emph{spherical categories}.  A spherical
category is a pivotal category with the extra property that left traces
equal right traces, which allows the evaluation of a graph $w$ embedded in
the sphere rather than in the plane.  A ribbon category is both pivotal and
braided in a compatible way, and allows the evaluation of a framed graph $w$
in $\R^3$.  We will only need the pivotal category axioms in this article,
but all categories considered are actually ribbon or compact closed.

\subsection{Sign conventions}
\label{s:signs}

In many cases a pivotal category $\cC$ which is linear over a field can
be modified to a different category $\cC'$.  We will be interested in two
modifications:  Sign changes to the pivotal structure of $\cC$ that do
not affect its tensor structure, and sign changes to the tensor structure
of $\cC$.  We want to restrict attention to those sign changes that allow
us to say that $\cC$ and $\cC'$ have the same algebraic information.  For
simplicity, when discussing signs, we assume that $\cC$ is abelian-linear
over an algebraically closed field $k$ not of characteristic 2, and
semisimple with irreducible trivial object.

Another objective of this section is to correctly interpreted labelled
graphs $w$ in a pivotal category with unoriented edges.  Some edges in
a labelled graph $w$ for a pivotal or compact closed category can be
unoriented.  Suppose that $A \cong A^*$ is self-dual, and suppose further
that the isomorphism $\phi_A \in \Hom(A,A^*)$ is cyclically invariant if
interpreted as an element of $\Inv(A \tensor A)$.  In this case we say
that $A$ is \emph{symmetrically self-dual}.  (This definition
does not require any linearity assumption.)  Then an unoriented edge can
be defined by a replacement:
\eq{e:tag}{\begin{tikzpicture}[baseline=-.5ex]
\draw[web] (-.75,0) -- (.75,0) node[black,above,midway] {$A$};
\end{tikzpicture}\; \defeq\;
\begin{tikzpicture}[baseline=-.5ex]
\draw[web,midto] (-1,0) -- (0,0) node[black,above,midway] {$A$};
\draw[web,midto] (1,0) -- (0,0) node[black,above,midway] {$A$};
\fill[web] (0,0) circle (.07);
\end{tikzpicture},}
where the dot on the right side represents $\phi_A$.  Algebraically, if
(and only if) $A$ is symmetrically self-dual, then $\cC$ is equivalent
to a pivotal category in which $A = A^*$ outright, and $b_A = b_{A^*}$
and $d_A = d_{A^*}$.  If every self-dual object in $\cC$ is symmetrically
self-dual, then $\cC$ is called \emph{unimodal} \cite{Turaev:quantum}.
If $A \cong A^*$ but $A$ is not symmetrically self-dual, then only the
right side of \eqref{e:tag} makes sense, and only if it is altered in
some way to break symmetry; Morrison denotes such a morphism by a ``tag"
\cite{Morrison:diagram}.

Suppose instead that $A \cong A^*$ but $A$ is not symmetrically self-dual,
and suppose that $\cC$ is $k$-linear and semisimple and $A$ is irreducible.
Then by Schur's lemma, $\Hom(A,A^*)$ is 1-dimensional and rotation $R$
is multiplication by $-1$.  In this case, $A$ is \emph{anti-symmetrically
self-dual}.  Thus we can ask whether we can make $\cC$ unimodal by
changing signs.  This is what happens in our case (see \sec{s:examples}),
but there are also examples (namely, representation categories of finite
groups) that are not unimodal for any pivotal structure.

To understand the allowed sign changes to the pivotal structure of $\cC$,
we first assume by category equivalence that $A$ and $A^*$ are different
objects for every $A$.  Then by \eqref{e:plumbing}, we can negate $b_A$ and
$d_A$ for some irreducible $A$, without changing $b_{A^*}$ and $d_{A^*}$.
This yields a new pivotal category $\cC'$, provided that the sign change
function $s(A)$ satisfies
$$s(A) = s(B)s(C)$$
whenever $\Hom(A,B \tensor C) \ne 0$.  If $A$ is self-dual
and $s(A) = -1$, then this modification changes the sign
of the self-duality of $A$.  It also negates the dimension of $A$,
by definition
$$\dim(A) = d_{A^*} \circ b_A.$$
Finally, since we are changing the pivotal structure by signs rather than
by other phases, $\cC$ is spherical if only if $\cC'$ is spherical.

We can change the sign of the tensor structure of $\cC$ by a similar
but more complicated construction.  We can assume, after passing to an
equivalent category, that the objects $\cC$ are a free polynomial semiring
over the irreducible objects of $\cC$ with respect to the operations $\oplus$
and $\tensor$.  If
$$A = A_1 \tensor A_2 \tensor \cdots \tensor A_a$$
is a tensor product of irreducibles, and likewise $B$, $C$, and $D$
are also tensor products of irreducibles, then we can
change the sign of the tensor product map
$$\tensor:\Hom(A,B) \tensor \Hom(C,D) \to \Hom(A \tensor C,B \tensor D)$$
by some sign function $s(A,B;C,D) \in \{\pm 1\}$, defined when
$\Hom(A,B)$ and $\Hom(C,D)$ are both nonzero.  In order for the result
$\cC'$ to be another pivotal category, we need to check that compositions
and tensor products of morphisms are still both associative.  In other
words, we need to check the equations
\begin{align*}
s(A,C;D,F) &= s(A,B;D,E)s(B,C;E,F) \\
s(A,B;C,D)s(A \tensor C,B \tensor D;E,F)
    &= s(A,B;C \tensor E,D \tensor F)s(C,D;E,F)
\end{align*}
when the right sides are defined.  It turns out that if $\cC$
is pivotal or spherical, then the new tensor category $\cC'$ can also be
made pivotal or spherical.

\subsection{Examples}
\label{s:examples}

A fundamental example of a pivotal category, indeed a compact closed
category, is the category $\vect(k)$ of finite-dimensional vector spaces
over a field $k$.  In this example, a web can be interpreted as the graph
of a tensor calculus expression (or a ``spin network").  For example,
if $\eps_{abc}$ is a trilinear determinant form on a 3-dimensional vector
space $V$, and if $\eps^{abc}$ is the dual form on $V^*$, then the tensor
$\eps_{abc} \eps^{cde}$ (with repeated indices summed) can be drawn as
$$\begin{tikzpicture}[baseline=-.5ex,scale=.65]
\draw[midto,web] (-.5,.866) -- (0,0);
\draw[midfrom,web] (0,0) -- (1,0);
\draw[midto,web] (1,0) -- (1.5,.866);
\draw[midto,web] (-.5,-.866) -- (0,0);
\draw[midto,web] (1,0) -- (1.5,-.866);
\node[anchor=south east] at (-.5,.866) {$a$};
\node[anchor=north east] at (-.5,-.866) {$b$};
\node[anchor=south west] at (1.5,.866) {$d$};
\node[anchor=north west] at (1.5,-.866) {$e$};
\end{tikzpicture},$$
with the convention in this case that the vertex labels can be inferred
from context.  If the characteristic of $k$ is not 2, then another
fundamental example is the category $\svect(k)$ of finite-dimensional
\emph{super vector spaces}, which are $\Z/2$-graded vector spaces with
a non-trivial symmetric and pivotal structure.   Namely, if $v \in V$
and $w \in W$ are homogeneous elements of super vector spaces, then
$$c_{V,W}(v \tensor w) = (-1)^{(\deg v)(\deg w)} w \tensor v.$$
If $v \in V$ and $w \in V^*$ are homogeneous, then the cap $d_V$ is likewise
adjusted so that
$$d_V(v \tensor w) = (-1)^{(deg v)} w(v).$$

If $G$ is a group (or a Lie group, Lie algebra, or algebraic
group), then $\rep(G,k)$, the category of finite-dimensional representations
(or continuous or algebraic representations) over $k$ is a pivotal category
with a pivotal functor to $\vect(k)$.  For the remainder of the article,
we let $G$ be a simple, simply connected algebraic group over $\C$ (and
later we will specialize to $G = \SL(3)$).  We will study the pivotal
category $\rep(G) = \rep(G,\C)$.

There is a deformation $\rep_q(G)$ of $\rep(G) = \rep_1(G)$ that consists
of representations of the quantum group $U_q(\fg)$, when the parameter $q$
is not a root of unity.  (The deformation also exists when $q$ is a root of
unity, but there is more than one standard choice for it.)  This deformation
is also a pivotal category, although it has no pivotal functor to $\vect$,
because the cup and cap morphisms deform.  Even though many ideas in this
article are clearly related to quantum representations, we will concentrate
on $\rep(G)$, except in \sec{s:euler} when $\rep_{-1}(G)$ will
also appear.

We are interested in two other variations of $\rep(G)$.  First, we want
to change its pivotal structure to make it unimodal.  Recall that the
irreducible representations $V(\lambda)$ of $G$ are labelled by the set of
dominant weights.  For a dominant weight $\lambda$, we write $\lambda^*$ for
the dominant weight such that $V(\lambda)^* \cong V(\lambda^*)$.  We also
write $\rho$ for the Weyl vector and $\rhov$ for the dual Weyl vector.
We make each $V(\lambda)$ a super vector space by giving it the grading
$\braket{2\lambda,\rhov} \bmod 2$.  In this way we realize $\rep(G)$ as a
subcategory of $\svect(\C)$ with a different pivotal and symmetric structure,
and we call this version $\rep^u(G)$.  Likewise, it has a unimodal pivotal
deformation $\rep^u_q(G)$.

Following \sec{s:signs}, $\rep(G)$ and $\rep_{-1}(G)$ differ only by sign
rules and all of them have equivalent information.  To obtain $\rep_{-1}(G)$
from $\rep(G)$ in this fashion, we use the abbreviations
\begin{align*} \vlambda &= (\lambda_1,\lambda_2,\ldots,\lambda_n) \\
V(\vlambda) &= V(\lambda_1) \tensor \cdots \tensor V(\lambda_n) \\
\lambda &= \sum_i \lambda_i.
\end{align*}
Then we define the sign rule
$$s(V(\vlambda),V(\vgamma);V(\vmu),V(\vnu)) =
    (-1)^{\braket{2\lambda,\rhov}\braket{\mu-\nu,\rhov}}.$$
This sign rule takes $\rep(G)$ to $\rep_{-1}(G)$ and $\rep^u(G)$ to
$\rep_{-1}^u(G)$.

The other variation is a restriction to minuscule representations.
Recall that a dominant weight $\lambda$ is called \emph{minuscule} if
$\braket{\alphav, \lambda} \le 1$ for every positive coroot $\alphav$.
If $\lambda$ is a minuscule dominant weight, then $V(\lambda)$ is called a
\emph{minuscule representation}.  These representations have the special
property that all of their weights are in the Weyl orbit of the highest
weight.  We define $\rep(G)_\min$ to be the monoidal subcategory of $\rep(G)$
generated by minuscule representations.  So the objects of $\rep(G)_\min$
are tensor products of minuscule representations.  It is a symmetric
category which is neither an additive nor an abelian category.  If there
exists a minuscule $\lambda$ such that $\braket{2\lambda, \rhov}$ is odd,
then $\rep(G)_\min$ is also not a pivotal category, because it is skeletal
and yet has objects which are anti-symmetrically self-dual in $\rep(G)$.
However, $\rep^u(G)_\min$ is a well-defined pivotal category in which $*$
is a strict involution and $V(\lambda^*) = V(\lambda)^*$.

In the case $G = \SL(n)$, and in some other cases, $\rep(G)$ can be recovered
as the Karoubi envelope of $\rep(G)_\min$, although we will not use this
construction in this article.

The other main pivotal category which we will study in this paper is the
category of $\Gv(\cO)$-equivariant perverse sheaves $\perv(\Gr)$ on $\Gr$.
This category has a relatively straightforward pivotal structure. It also
has a more delicate symmetric structure which is called a ``commutativity
constraint'' or ``braiding'', as defined by Ginzburg \cite{Ginzburg:loop}
and Mirkovi\'c-Vilonen \cite{MV:geometric} in two different ways.
(See also \cite[Sec. 5.3.8]{BD:hitchin}.)
\thm{th:satake} states that $\perv(G)$ is equivalent to $\rep^u(G)$ both as a
pivotal category and a symmetric category; we will be more interested in
the pivotal structure.  We also will be more interested in the minuscule
analog of $\perv(\Gr)$, which we will analyze in \sec{s:conv}.

\subsection{Free spiders and presentations}

Pivotal categories can also be presented by generators and relations.
If the pivotal category is additive-linear over a ring or a field, then
it can presented in the same sense, using linear combinations of words in
the generators.  In general there are generating objects (or edges) and
generating morphisms (or invariants or vertices), while the relations are
all morphisms.  Relations in a pivotal category are also known as
\emph{planar skein relations}.

We now define the free spider $\fsp(G)$ to be the free $\C$-linear pivotal
category generated by an edge for each minuscule representation of $G$
and a vertex for every triple $\lambda, \mu, \nu$ of minuscule dominant
weights such that
$$\Inv_G(V(\lambda, \mu, \nu)) \ne 0.$$
Note that the minuscule condition
forces this vector space to be at most one-dimensional.  In $\fsp(G)$,
we also impose that the dual of the $\lambda$ edge is $\lambda^*$.
In \cite{Morrison:diagram}, $\fsp(\SL(n))$ was denoted $\text{Sym}_n$.

A free spider has the same relationship to webs as a free group has to
words in its generators.  Namely, two webs are equal in $\fsp(G)$ if and
only if they are isotopic rel boundary.  (Selinger \cite{Selinger:survey}
also defines free categories of various kinds generated by signatures.)

Let us fix $q \in \C$, non-zero and not a root of unity (but possibly
equal to 1).  There is a pivotal functor
$$\Psi:\fsp(G) \to \rep^u_q(G)_\min,$$
which is defined by choosing a non-zero element in each invariant space
$$\Inv_{U_q(\fg)}(V(\lambda, \mu, \nu)).$$
In particular, for each web $w$ with boundary $\vlambda$, we obtain
an element
$$\Psi(w) \in \Inv_{U_q(\fg)}(V(\vlambda)).$$
Actually, since webs are a notation for words in any pivotal category,
we could say also say that $w$ ``is" $\Psi(w)$, or that its value is
$\Psi(w)$.  But the distinction between $w$ and $\Psi(w)$ will be useful
for us.  The first result is that $\Psi$ is surjective when $G = \SL(n)$
\cite[Prop. 3.5.8]{Morrison:diagram}. (This follows from Weyl's fundamental
theorem of invariant theory.)  Thus, the vectors $\Psi(w)$ of webs $w$
span the invariant spaces.

It is an open problem to generate the kernel of $\Psi$ with planar skein
relations in $\fsp(G)$.  This problem has been solved when $G$ has rank 1
or 2 by the third author \cite{Kuperberg:spiders}.  Kim \cite{Kim:thesis}
has conjectured an answer for $\SL(4)$ in \cite{Kim:thesis} and Morrison
\cite{Morrison:diagram} has done so for $\SL(n)$.  Once these planar
skein relations (which must depend on $q$) are determined, then the
resulting presented pivotal category can be called a spider and we denote
it $\spd_q(G)$.

We now review the known solutions for $\SL(2)$ and $\SL(3)$.  The
\emph{Temperley-Lieb category} or \emph{$A_1$ spider} $\spd_q(\SL(2))$ is
the quotient of $\fsp(\SL(2))$ by the single relation
\eq{e:tl}{\begin{tikzpicture}[baseline=-.5ex,web]
\draw (0,0) circle (.4);
\end{tikzpicture} \;\;= \;\;-q-q^{-1}.}
(Since $\SL(2)$ has a single, self-dual minuscule representation,
$\fsp(\SL(2))$ and $\spd_q(\SL(2))$ have unoriented edges with a single
color or label.)  The \emph{$A_2$ spider} $\spd_q(\SL(3))$ is the quotient
of $\fsp(\SL(3))$ by the relations
\begin{align}
\begin{tikzpicture}[baseline=-.5ex,web,allow upside down]
\draw[midto,rotate=180] (0,0) circle (.4);
\end{tikzpicture}\;\; &= \;\; q^2+1+q^{-2} \nonumber \\
\begin{tikzpicture}[baseline=-.5ex,web]
\draw[midto] (-1.25,0) -- (-.5,0);
\draw[midto] (.5,0) -- (1.25,0);
\draw (.5,0) arc (30:150:.577); \draw[midto] (0,.288) -- (-.01,.288);
\draw (.5,0) arc (330:210:.577); \draw[midto] (0,-.288) -- (-.01,-.288);
\end{tikzpicture}\;\; &= \;\;
(-q-q^{-1})\;\;
\begin{tikzpicture}[baseline=-.5ex,web]
\draw[midto] (-.5,0) -- (.5,0);
\end{tikzpicture} \label{e:a2spider} \\
\begin{tikzpicture}[baseline=-.5ex,web]
\draw (.35,.35) arc (60:120:.7) arc (150:210:.7)
    arc (240:300:.7) arc (-30:30:.7);
\draw[midfrom] (0,.444) -- (.01,.444);
\draw[midfrom] (-.444,0) -- (-.444,-.01);
\draw[midfrom] (0,-.444) -- (-.01,-.444);
\draw[midfrom] (.444,0) -- (.444,.01);
\draw[midto] (.35,.35) -- (.707,.707);
\draw[midfrom] (-.35,.35) -- (-.707,.707);
\draw[midto] (-.35,-.35) -- (-.707,-.707);
\draw[midfrom] (.35,-.35) -- (.707,-.707);
\end{tikzpicture}\;\; &=\;\;
\begin{tikzpicture}[baseline=-.5ex,web]
\draw (.5,.5) arc (315:225:.707);
\draw[midto] (0,.293) -- (.01,.293);
\draw (-.5,-.5) arc (135:45:.707);
\draw[midto] (0,-.293) -- (-.01,-.293);
\end{tikzpicture}\;\; + \;\
\begin{tikzpicture}[baseline=-.5ex,web]
\draw (.5,-.5) arc (225:135:.707);
\draw[midto] (.293,0) -- (.293,.01);
\draw (-.5,.5) arc (45:-45:.707);
\draw[midto] (-.293,0) -- (-.293,-.01); \nonumber
\end{tikzpicture}\;\;.
\end{align}
(Since $\SL(3)$ has two minuscule representations which are dual to each
other, $\fsp(\SL(3))$ and $\spd_q(\SL(3))$ have oriented edges with
one label or color.  By convention, the edge is labelled by the first
fundamental representation $\omega_1$ in the direction that it is oriented.)
The other two known spiders, $\spd_q(B_2)$ and $\spd_q(G_2)$, have similar
but more complicated presentations.

\begin{theorem}[Kauffman \cite{Kauffman:spinknot}] If $q$ is not a root
of unity, then $\spd_q(\SL(2))$ is equivalent to the pivotal category
$\rep^u_q(\SL(2))_\min$ of minuscule representations.
\label{th:kauffman} \end{theorem}

\begin{theorem} \cite{Kuperberg:spiders} If $q$ is not a root of unity,
then $\spd_q(\SL(3))$ is equivalent to the pivotal category
$$\rep_q(\SL(3))_\min = \rep^u_q(\SL(3))_\min$$
of minuscule representations.
\end{theorem}
(In the case of $\SL(3)$, it turns out that $\rep_q(\SL(3))$
and $\rep^u_q(\SL(3))$ are the same; see \sec{s:examples}.)

A main property of the spider relations \eqref{e:a2spider} is that they
are confluent or Gr\"obner type.  In the free pivotal category generated
by the generating edges and vertices, each web can be graded by the
number of its faces.  Then each relation has exactly one leading term, an
elliptic face.  (In the $A_2$ spider, a face is \emph{elliptic} if it has
fewer than six sides.  In the other two rank 2 spiders, a face is elliptic
if the total angle of the corresponding dual vertex is less than $2\pi$,
so that the vertex is $\CAT(0)$; see \sec{s:diskoids}.)  A web that has
that face can be expressed, modulo the relation, as a linear combination of
lower-degree webs.  The Gr\"obner property, proved using a diamond lemma,
is that any two sequences of simplifications of the same web lead to the
same final expression.  This means that the webs that cannot be simplified,
\ie, the webs without elliptic faces or the non-elliptic webs, form a basis
of each invariant space.  There is an extended version of this result,
but we will restrict our attention to the minuscule case, summarized in
the following theorem.

\begin{theorem} \cite{Kuperberg:spiders} If $\vlambda$ is a sequence of
dominant minuscule weights of $\SL(3)$, then the non-elliptic type $A_2$
webs with boundary $\vlambda$ are a basis of $\Inv(V(\vlambda))$.
\label{th:sl3basis} \end{theorem}

\thm{th:unitriang} implies \thm{th:sl3basis} as a corollary.  However,
it is much more complicated than other proofs of \thm{th:sl3basis}
\cite{Westbury:trivalent,Kuperberg:notdual}.

\section{Affine geometry}
\label{s:ageom}

\subsection{Weight-valued metrics and linkages}
\label{s:wmetric}

In the usual definition of a metric space, distances take values in
the non-negative real numbers $\R_{\ge 0}$.  However, Kapovich, Leeb,
Millson \cite{KLM:generalized} have a theory of metric spaces in
which distances take values in the dominant Weyl chamber of $G$.  Two of
the axioms of such a generalized metric space are easy to state:
$$d(x,x) = 0 \qquad d(x,y) = d(y,x)^*.$$
The third axiom, the triangle inequality, is different.  The main results
of Kapovich, Leeb, and Millson are generalized triangle inequalities that
are satisfied in buildings and generalized symmetric spaces.  On the one
hand, the triangle inequalities in the $A_1$ case are the usual triangle
inequality.  On the other hand, the inequalities in higher rank cases are
decidedly non-trivial.

In this article, we will adopt the viewpoint of weight-valued metric spaces
in order to discuss isometries and distance comparisons.  We will not need
the generalized triangle inequalities, but we will need isometries and
distance comparisons.  The definition of an isometry is straightforward.
As for distance comparisons, we will say that $\mu \le \lambda$ as a
distance if and only if $\mu \le \lambda$ in the usual partial order on
dominant weights, namely that $\lambda - \mu$ is a non-negative integer
combination of simple roots.  Thus, a ball of radius $\lambda$ is then a
finite union of spheres of radius $\mu \le \lambda$.  For one construction
we will define distances that take values in the dominant Weyl chamber,
instead of integral weights; and then we say that $\mu \le \lambda$ when
$\lambda - \mu$ is a non-negative real combination of simple roots.

In addition to isometries, we will be interested in partial isometries in
which only some distances are preserved.  For this purpose, we define a
\emph{linkage} to be an oriented graph $\Gamma$ whose edges are labelled by
dominant weights.  As with webs, an edge labelled by $\lambda$ is equivalent
to the opposite edge labelled by $\lambda^*$.  Let $v(\Gamma)$ be the set of
vertices of $\Gamma$.  Then one may attempt to define a distance $d(p,q)$
between any two points $p,q \in v(\Gamma)$ by taking the shortest total
distance of a connecting path.  However, since weights are only partially
ordered, this minimum may not be unique.  We will say that $\Gamma$ has
\emph{coherent geodesics} if the minimum distance $\min(d(p,q))$ between
any two vertices $p$ and $q$ is unique, and if that minimum distance is
the length of the edge $(p,q)$ when $\Gamma$ has that edge.  In this case
$\Gamma$ can be completed to another linkage $\Gamma_g$ which is a complete
graph, using all distances as weights.

\subsection{Configuration spaces}
\label{s:config}

Let $X$ be a weight-valued metric space, and let $\Gamma$ be a linkage as
in \sec{s:wmetric}.  Let $v(\Gamma)$ be the set of vertices of $\Gamma$.
Then we define the \emph{linkage configuration space} $Q(\Gamma,X)$ to be
the set of maps
$$f:v(\Gamma) \to X$$
such that $d(f(p),f(q))$ equals the weight of the edge from $p$ to
$q$, when there is such an edge.  If $X$ and $\Gamma$ both have a base
point, then $Q(\Gamma,X)$ is instead the configuration space of based maps.
Another possibility is that $\Gamma$ has a base edge of length $\lambda$
and $X$ has two base points at distance $\lambda$; then $Q(\Gamma,X)$
is again the configuration space of based maps.  We will be interested in
four types of linkages $\Gamma$:
\begin{description}
\item[1] A path or \emph{polyline}.
\item[2] A cycle or polygon.
\item[3] The 1-skeleton $\Gamma(D)$ of a tiled diskoid $D$ (\sec{s:diskoids})
with edges labelled by weights.
\item[4] The complete linkage $\Gamma_g(D)$, if $\Gamma(D)$ has coherent
geodesics.
\end{description}

There is one final type of configuration space that is sometimes useful.
If an edge $(p,q)$ has weight $\lambda$, then we can ask that
$$d(f(p),f(q)) \le \lambda$$
instead of
$$d(f(p),f(q)) = \lambda.$$
The result is the contractive configuration space $Q_c(\Gamma,X)$.

Suppose that $X = G/H$ for some group $G$ with a subgroup $H$, and that
each sphere $X(\lambda)$ around the base point is a double coset of $H$.
Let $\Gamma$ be a linkage and let $\Gamma_0$ be the same linkage with a
chosen base point $0$.  Then there is a fibration
$$Q(\Gamma_0,X) \longto Q(\Gamma,X) \longto X.$$
Similarly, if $\Gamma_e$ denotes the same linkage with a base edge $e$
of length $\lambda$ incident to $0$, then there is also a fibration
\begin{equation} \label{eq:basededge}
Q(\Gamma_e,X) \longto Q(\Gamma_0,X) \longto X(\lambda),
\end{equation}
where $X(\lambda) = Q(\lambda,X)$ is the sphere of radius $\lambda$ around
the (first) base point of $X$, and the second base point is an arbitrary
point in $X(\lambda)$.

If $f:\Gamma_2 \to \Gamma_1$ is a map between linkages, then there
is a restriction map,
\eq{e:pi}{\pi_{\Gamma_2}^{\Gamma_1}:Q(\Gamma_1,X) \to Q(\Gamma_2,X)}
between their configuration spaces.  We will be particularly interested
in this map when $\Gamma_1$ is a sublinkage of $\Gamma_2$ (for example
its boundary).

Suppose now that $\Gamma = \Gamma_1 \cup \Gamma_2$, and that $\Gamma_1
\cap \Gamma_2$ is either an edge or a vertex.  If we base $\Gamma_2$
(but not $\Gamma_1$) at this intersection, then the configuration space
$Q(\Gamma,X)$ is a twisted product:
$$Q(\Gamma,X) = Q(\Gamma_1,X) \twistedprod Q(\Gamma_2,X).$$
Informally, $\Gamma_2$ is either an arm attached to $\Gamma_1$ at a point
which can swing freely in any direction, or a flap attached to $\Gamma_1$
along a 1-dimensional hinge which can swing freely in the remaining
directions.

\subsection{Diskoids}
\label{s:diskoids}

Recall that a \emph{piecewise-linear diskoid} is a contractible, compact,
piecewise-linear region in the plane.  (We will not need diskoids that are
not piecewise-linear.  But if one were to consider them, the most natural
definition could be to make it a planar, cell-like continuum.)  Any diskoid
$D$ has a polygonal boundary $P$ with a boundary map $P \to D$, which however
is not an inclusion unless $D$ is either a point or a disk.  \fig{f:diskoid}
shows an example of a diskoid $D$ with its boundary $P$.

\begin{fullfigure}{f:diskoid}{A diskoid $D$ with boundary $P$}
\begin{tikzpicture}[scale=.5]
\coordinate (p1) at (-2,.5);    \coordinate (p2) at (-3,2);
\coordinate (p3) at (-3.5,0);   \coordinate (p4) at (0,0);
\coordinate (p5) at (.5,-2);    \coordinate (p6) at (2.5,-.5);
\coordinate (p7) at (2,1);      \coordinate (p8) at (-1,-3);
\coordinate (p9) at (1.5,-3);   \coordinate (p10) at (3,2.5);
\coordinate (p11) at (2.98,2.52);
\draw[web,double distance = .2cm]
    (p1) -- (p2) -- (p3) -- (p1) -- (p4) -- (p5) -- (p8) -- (p9) -- (p5) --
    (p6) -- (p7) -- (p10) -- (p11) -- (p7) -- (p4) -- cycle;
\draw[fill=lightgray]
    (p1) -- (p2) -- (p3) -- (p1) -- (p4) -- (p5) -- (p8) -- (p9) -- (p5) --
    (p6) -- (p7) -- (p10) -- (p7) -- (p4) -- cycle;
\node at (1.25,-.375) {$D$};
\node[anchor=south west] at (-1,.5) {$P$};
\end{tikzpicture}
\end{fullfigure}

Note that since a diskoid comes with an embedding in the plane, its boundary
$P$ is implicitly oriented, so that the edges of $P$ are cyclically ordered.
We will assume a clockwise orientation in this article.  Trees are diskoids,
and \fig{f:a1example} has an example of the polygonal boundary of a tree;
the polygon traverses each edge twice.

A diskoid $D$ can be tiled by polygons. Formally, a \emph{tiling} of $D$
is a piecewise-linear CW complex structure on $D$ with embedded 2-cells.
If $D$ is decorated in this way, then we define the graph $\Gamma(D)$ to be
its 1-skeleton.  Then, as above, $\Gamma(D)$ can be made into a linkage,
which means, explicitly, that the edges of $D$ are labelled by distances.
In this article we will not need to the label the faces (or 2-cells)
of a tiled diskoid to define its configuration space, but only because
the corresponding representation theory is multiplicity-free.  In future
work, the faces could also be labelled in order to define more restrictive
configuration spaces.  We will write $Q(D)$ for $Q(\Gamma(D))$ and $Q_g(D)$
for $Q(\Gamma_g(D))$.

In some cases, although not the most important cases, we will be interested
in diskoids with bubbles.  By definition, a \emph{diskoid with bubbles} is,
inductively, either a diskoid, or a one-point union of a smaller diskoid
with bubbles and either a line segment or a piecewise linear 2-sphere.
The extra line segments and 2-spheres are not embedded in the plane and
do not affect the boundary of the diskoid, even if the attachment point is
on the boundary.  The discussion of the previous paragraph applies equally
well to diskoids with bubbles.

Our interest in diskoids arises from the fact that they are geometrically
dual to webs.  As in the introduction, let $w$ be a web in $\fsp(G)$ with
boundary $\vlambda$.  Then it has a dual diskoid $D = D(w)$, with bubbles
if $w$ has closed components, and with a natural base point. To be precise,
$D$ has a vertex for every internal or external face of $w$; two vertices
are connected by an edge when the faces of $w$ are adjacent; and there
is a triangle glued to three edges whenever the dual edges of $w$ meet at
a vertex.  We label the edges of $D$ using the labels of the corresponding
edges of $w$; also, if an edge of $w$ is oriented, we transfer it to an
orientation of the dual edge of $D$ by rotating it counterclockwise.  As
a result, the boundary of the diskoid $D$ is the polygon $P(\vlambda)$.
\fig{f:a1example} shows an example of an $A_1$ web and its dual diskoid,
which in the $A_1$ case is always a tree.  \fig{f:a2example} shows an
example of an $A_2$ web and its dual diskoid, which happens to be a disk
because the corresponding web is connected.

In this construction, $D$ is always triangulated because $w$ is always
trivalent.  The vertices of $D$ are a weight-valued metric space, and by
linear extension the whole of $D$ is a Weyl-chamber-valued metric space.
We can also simplify this metric to an ordinary metric space by taking the
Euclidean length of the vector-valued distance.  Finally, suppose that
$w$ is an $A_2$ web (or a $B_2$ or $G_2$ web). Then $w$ is non-elliptic
if and only if $D$, in its ordinary metric, is $\CAT(0)$ in the sense
of Gromov \cite{Gromov:hyperbolic}.  This follows from the fact that $D$
is contractible and the condition that all complete angles in $D$ are at
least $2\pi$.

\subsection{Affine Grassmannians and buildings}

As before, let $G$ be a simple, simply-connected complex algebraic group
and let $\Gv$ be its Langlands dual group.  Let $\cO = \C[[t]]$ be the
ring of formal power series over $\C$ and let $\cK = \C((t))$ be its
fraction field.  Then
$$\Gr = \Gr(\Gv) = \Gv(\cK)/\Gv(\cO)$$
is the \emph{affine Grassmannian} for $\Gv$ with residue field $\C$.  It is
an ind-variety over $\C$, meaning that it is a direct limit of algebraic
varieties (of increasing dimension).  The affine Grassmannian $\Gr$ is
also a weight-valued metric space:  The double cosets $\Gv(\cO)\backslash
\Gv(\cK)/\Gv(\cO)$ are bijective with the cone $\Lambda_+$ of dominant
coweights of $\Gv$, which is the same as the cone of dominant weights of $G$.
More precisely, for each coweight $\mu$ of $\Gv$, there is an associated
point $t^\mu$ in the affine Grassmannian.  If $p,q$ are two arbitrary
points of the affine Grassmannian, then we can find $g \in \Gv(\cK)$
such that $gp = t^0$ and $gq = t^\mu$ for some unique dominant
coweight $\mu$.  Under this circumstance, we write $d(p,q) = \mu$.
So the action of $\Gv(\cK)$ preserves distances and $d(t^0, t^\mu) =
\mu$ for any dominant weight $\mu$.

The affine Grassmannian $\Gr$ is also a subset of the vertices $\Gr' =
v(\Delta)$ of an associated simplicial complex called an \emph{affine
building} $\Delta = \Delta(\Gv)$ \cite{Ronan:buildings} whose type
is the extended Dynkin type of $\Gv$.   The simplices of this affine
building are given by parahoric subgroups of the affine Kac-Moody group
$\widehat{\Gv}$.  For a detailed description of affine buildings from
this perspective, see \cite{GL:cycles}.

An affine building $\Delta$ satisfies the following axioms:
\begin{description}
\item[1] The building $\Delta$ is a non-disjoint union of \emph{apartments},
each of which is a copy of the Weyl alcove simplicial complex of $\Gv$.

\item[2] Any two simplices of $\Delta$ of any dimension are both contained
in at least one apartment $\Sigma$.

\item[3] Given two apartments $\Sigma$ and $\Sigma'$ and two simplices
$\alpha,\alpha' \in \Sigma \cap \Sigma'$, there is an isomorphism $f:\Sigma
\to \Sigma'$ that fixes $\alpha$ and $\alpha'$ pointwise.
\end{description}

The axioms imply that the vertices of $\Delta$, denoted $\Gr'$, are
canonically colored by the vertices of the extended Dynkin diagram $\hat{I}
= I \sqcup \{0\}$ of $\Gv$, or equivalently the vertices of the standard
Weyl alcove $\delta$ of $\Gv$.   Moreover, every maximal simplex of
$\Delta$ is a copy of $\delta$; it has exactly one vertex of each color.
The affine Grassmannian consists of those vertices colored by $0$ and by
minuscule nodes of the Dynkin diagram of $\Gv$.

The axioms also imply that $v(\Delta)$, and more generally the realization
$|\Delta|$ of $\Delta$, have a metric taking values in Weyl chamber.
(But not necessarily integral weights as one sees in $\Gr$.)  Namely, if
$p,q \in |\Delta|$, then $p,q \in |\Sigma|$ for an apartment $\Sigma$, and
after a suitable automorphism $p = q + \lambda$ for some vector $\lambda$ in
the dominant Weyl chamber.  We then define $d(p,q) = \lambda$.  (The metric
has coherent geodesics, and it extends the metric defined above for $\Gr$.)
We will need the following fact.

\begin{lemma} If $p,q \in |\Delta|$, then every geodesic path $\gamma$
from $p$ to $q$ is contained in every apartment $\Sigma$ such that $p,q
\in |\Sigma|$.
\end{lemma}

A subtle feature of the above affine building $\Delta$ is that it has two
very different geometries.  As an ordinary simplicial complex, its vertex
set $\Gr'$ is discrete, and $\Gr'$ has a combinatorial, weight-valued metric.
The vertex set $\Gr'$ is also naturally an algebraic ind-variety over $\C$,
as is the set of vertices of any given color or the set of simplices of
$\Delta$ of any given type.  This second geometry endows $\Gr'$ with both
a Zariski topology and an analytic topology.  Among the relations between
these two geometries, we will need the following fact.

\begin{proposition} The algebraic-geometric closure $\bar{\Gr'(\lambda)}$
of the sphere $\Gr'(\lambda)$ of radius $\lambda$ is the set of all points
in the metric ball of radius $\lambda$ that have the same color as $\lambda$.
\label{p:sphereball} \end{proposition}

An affine building $\Delta$ has a third geometry which is related to
the weight-valued metric but is not the same.  Namely, we can give the
Weyl alcove $\delta$ its standard Euclidean structure, and consider the
induced metric on the realization $|\Delta|$ of $\Delta$.  This locally
Euclidean metric can also be defined as $||d(p,q)||_2$, where $d(p,q)$ is
the weight-valued metric on $|\Delta|$.

\begin{theorem}[Bruhat-Tits \cite{BT:local1}] Every affine building is a
$\CAT(0)$ space with respect to its locally Euclidean metric.
\label{th:cat0} \end{theorem}

If $G = \SL(n)$ and thus $\Gv = \PGL(n)$, then $\Gr = \Gr'$, and there
is a simple description of $\Delta$.  Namely, a finite set of vertices
in $\Gr$ subtends a simplex if and only if the distances between them are
all minuscule.

Finally, to close a circle, let $L(\vlambda)$ be a polyline whose sides
are labelled by
$$\vlambda = (\lambda_1,\lambda_2,\dots,\lambda_n),$$
based at the beginning.  Let $P(\vlambda)$ be the corresponding polygon,
based between $\lambda_n$ and $\lambda_1$.  Then the contractive polyline
configuration space
$$\Gr(\vlambda) = Q_c(L(\vlambda),\Gr)$$
is the domain of the convolution morphism.  The restriction map coming
from the projection onto the boundary $L(\vlambda) \to \pt$, or
$$\pi_\pt^{L(\vlambda)}:Q_c(L(\vlambda),\Gr) \to \Gr,$$
is the convolution morphism.  In keeping with the standard notation,
we will denote it by
$$m_\vlambda = \pi_\pt^{L(\vlambda)}.$$
Meanwhile the contractive polygon configuration space
$$Q_c(P(\vlambda),\Gr) = F(\vlambda) = m_\vlambda^{-1}(t^0)$$
is the Satake fibre.  As another bit of notation, if $\Gamma$ is a linkage,
we will elide the $\Gr$ and write $Q(\Gamma)$ for $Q(\Gamma,\Gr)$, etc.

\section{Geometric Satake for tensor products of minuscule representations}
\label{s:minuscule}

\subsection{Minuscule paths and components of Satake fibres}
\label{s:paths}

The full geometric Satake correspondence, \thm{th:satake}, simplifies
considerably when the weights are minuscule.  In this special case, Haines
\cite[Thm. 3.1]{Haines:equidim} showed that all components of $F(\vlambda)$
are of maximal dimension.  We can use his ideas to give an explicit
description of these components using minuscule paths.  In addition to
previous notation, let $W$ be the Weyl group of $G$.

Let $\lambda$ be a minuscule dominant weight.  Then there are no dominant
weights less than $\lambda$, so the sphere of radius $\lambda$ equals
the ball of radius $\lambda$.  Hence the sphere $\Gr(\lambda)$ is closed
in the algebraic geometry of $\Gr$ by \prop{p:sphereball}, and thus it is
projective and smooth.  In fact, $\Gv$ acts transitively on $\Gr(\lambda)$.
The stabilizer of $t^\lambda$ is $M(\lambda)$, the opposite maximal proper
parabolic subgroup corresponding to the minuscule weight $\lambda$.  Thus
$\Gr(\lambda)$ is isomorphic to the partial flag variety $\Gv/M(\lambda)$.

More generally, if $\Gamma$ is a \emph{minuscule linkage}, meaning that
all of its edges are minuscule, then
$$Q(\Gamma) = Q_c(\Gamma) = \bar{Q(\Gamma)}.$$

Let
$$\vlambda = (\lambda_1, \ldots, \lambda_n)$$
be a sequence of minuscule dominant weights.  A \emph{minuscule path}
(ending at 0) of type $\vlambda$ is a sequence of dominant weights
$$\vmu = (\mu_0,\mu_1,\mu_2,\ldots,\mu_n)$$
such that $\mu_k - \mu_{k-1} \in W\lambda_k$ for every $k$, and such that
$$\mu_0 = \mu_n = 0.$$
In other words, the $k$th step of the path $\vmu$ is a weight of
$V(\lambda_k)$, and the path is restricted to the dominant Weyl chamber
$\Lambda_+$.  Minuscule paths are a special case of Littelmann paths
\cite{Littelmann:paths}, but it was much earlier folklore knowledge that
the number of minuscule paths of type $\vlambda$ is the dimension of
$\Inv(V(\vlambda))$. (See Humphreys \cite[Ex. 24.9]{Humphreys:gtm}, and
use induction.)

\begin{fullfigure}{f:fan}{The fan diskoid $A(\vlambda,\vmu)$}
\begin{tikzpicture}[scale=.75,draw=darkred,semithick]
\draw[midto] (0,0) to node[below left=-.5ex] {$\lambda_1=\mu_1$} (-2,1);
\draw[midto] (-2,1) to node[below left] {$\lambda_2$} (-3,3);
\draw[midto] (-3,3) to node[above left] {$\lambda_3$} (-2,5);
\draw[midto] (-2,5) to node[above] {$\lambda_4$} (0,6);
\draw (0,6) -- (1,6);
\draw (3,3) -- (3,2);
\draw[midto] (3,2) to node[below right] {$\lambda_{n-1}$} (2,0);
\draw[midto] (2,0) to node[below right] {$\lambda_n = \mu^*_{n-1}$} (0,0);
\draw[midto] (0,0) to node[above right] {$\mu_2$} (-3,3);
\draw[midto] (0,0) to node[above right] {$\mu_3$} (-2,5);
\draw[midto] (0,0) to node[above right] {$\mu_4$} (0,6);
\draw[midto] (0,0) to node[above=.75ex] {$\mu_{n-2}$} (3,2);
\fill (2.14,5.04) circle (.035); \fill (2.3,4.8) circle (.035);
\fill (2.46,4.56) circle (.035);
\fill (1.04,3.04) circle (.035); \fill (1.2,2.8) circle (.035);
\fill (1.36,2.56) circle (.035);
\end{tikzpicture}
\end{fullfigure}

Given a minuscule path $\vmu$ of type $\vlambda$, we define a based diskoid
$A(\vlambda,\vmu)$ in the shape of a fan, whose the boundary is the polygon
$P(\vlambda)$ and whose ribs are labelled by $\vmu$, as in \fig{f:fan}.
Then there is a natural inclusion
$$Q(A(\vlambda,\vmu)) \subseteq F(\vlambda).$$
The following result is implicit in the work of Haines \cite{Haines:equidim}.

\begin{theorem} For each minuscule path $\vmu$, the fan configuration space
$Q(A(\vlambda,\vmu))$ is a dense subset of one component of $F(\vlambda)$.
The induced correspondence is a bijection between minuscule paths and
components of $F(\vlambda)$.
\label{th:haines} \end{theorem}

The key to the proof of this theorem is the following lemma.

\begin{lemma} Let
$$T_e(\mu,\lambda,\nu) \;\;=\;\;
\begin{tikzpicture}[baseline,semithick,darkred]
\draw[midto] (-1,0) -- (1,-0.5) node[black,midway,below] {$\mu$};
\draw[midto] (-1,0) -- (1.5,0.5) node[black,midway,above] {$\nu$};
\draw[midto] (1,-0.5) -- (1.5,0.5) node[black,midway,right] {$\lambda$};
\end{tikzpicture}$$
be a triangle with a minuscule edge $\lambda$, based at the edge $e$ of
length $\mu$.  Then $Q(T_e(\mu,\lambda,\nu))$ is non-empty if and only if
there exists $w \in W$ such that $\mu + w\lambda = \nu$.  If it is non-empty,
then it is smooth and has complex dimension $\braket{\nu-\mu+\lambda,\rhov}$.
\label{l:smooth} \end{lemma}

\begin{proof}
Let $W(\mu)$ denote the stabilizer of $\mu$ in the Weyl group.  It is a
parabolic subgroup of $W$.

Let us choose the base edge in $\Gr$ to be the edge connecting $t^{-\mu}$
and $t^0$.  Then the edge based configuration space $Q(T_e(\mu, \lambda,
\nu))$ is a subvariety of $\Gr(\lambda)$ since there is only one free
vertex.  In fact
$$Q(T_e(\mu, \lambda, \nu)) =
    \{ p \in \Gr(\lambda) | d(t^{-\mu}, p) = \nu \}.$$
Let $A$ denote the set $W/W(\lambda)$, which we regard as a poset
using the opposite Bruhat order.  With this order, $A$ becomes
the poset of $B$-orbits on $\Gr(\lambda) = \Gv/M(\lambda)$, where $B$
is the Borel subgroup of $\Gv$.  We will be interested in the action of
$W(\mu)$ on $A$ by left multiplication.  The quotient $W(\mu) \setminus
A$ is the set of $M_+(\mu)$ orbits on $\Gr(\lambda)$, where $M_+(\mu) =
\mathrm{Stab}_\Gv(t^{-\mu})$ is the parabolic subgroup corresponding to
the minuscule weight $\mu$.

Hence we can write any point $p$ of $\Gr(\lambda)$ as $p = g t^{a \lambda}$
where $g \in M_+(\mu)$ and $a \in A$ is chosen to be a maximal length
representative for the orbit of $W(\mu)$.  The action of
$M_+(\mu)$ on $\Gr$ stabilizes $t^{-\mu}$ so
$$d(t^{-\mu}, gt^{a \lambda})
    = d(t^{-\mu}, t^{a\lambda}) = d(t^0, t^{\mu + a \lambda}).$$
Now, we claim that $\mu + a \lambda$ is always dominant.  Let us write
$a = [w]$ for $w \in W$.  We must check that
$$\braket{\mu + w\lambda, \alphav_i} =
    \braket{\mu, \alphav_i} + \braket{\lambda, w \alphav_i} \ge 0$$
for all simple coroots $\alphav_i$.  We break this calculation into
two cases.

First, suppose that $s_i \mu = \mu$.  Then $\braket{\mu, \alphav_i} = 0$.
On the other hand $s_i w > w$ (in the usual Bruhat order) by the maximality
of $a$ in the $W(\mu)$-orbit.  This implies that $w \alphav_i$ is a
positive coroot, which implies that $\braket{\lambda, w \alphav_i }$
is non-negative (since $\lambda$ is dominant). Hence
$$\braket{\mu, \alphav_i} + \braket{\lambda, w \alphav_i} \ge 0.$$

Next, suppose that $s_i \mu \ne \mu$.  Then since $\mu$ is dominant,
$\braket{\mu, \alphav_i} \ge 1$.  On the other hand, $|\braket{ \lambda,
w \alphav_i}| \le 1$ since $w \alphav_i$ is a coroot and $\lambda$
is minuscule.  Hence
$$\braket{\mu, \alphav_i} + \braket{\lambda, w \alphav_i} \ge 0$$
in this case as well.

Since $\mu + a\lambda$ is always dominant, we conclude that
$$d(t^{-\mu}, gt^{a \lambda}) = \mu + a\lambda.$$
Hence, $Q(T_e(\mu,\lambda,\nu))$ is non-empty iff there exists $w \in
W$ such that $\mu + w\lambda = \nu$. (The above argument shows that
$[w]$ will necessarily be a maximal length representative for the
$W(\mu)$ action on $A$.)  If such $w$ exists, then the configuration
space $Q(T_e(\mu,\lambda,\nu))$ is simply the $M(\mu)$-orbit through
$t^{w\lambda}$.  Hence it is smooth and its dimension is given by
the length of $[w]$ in $A$ because it is of the same dimension as the
$B$-orbit through $t^{w\lambda}$.  Since $\lambda$ is minuscule, this
equals $\braket{w \lambda + \lambda, \rhov}$ as desired.
\end{proof}

\begin{proof}[Proof of \thm{th:haines}] It is easy to show by induction
that the fan configuration space
$$Q(A(\vlambda,\vmu)) = Q(P_e(\mu_0,\lambda,\mu_1)) \twistedprod
    \cdots \twistedprod Q(P_e(\mu_{n-1},\lambda_n,\mu_n))$$
is an iterated twisted product of triangle configuration spaces.
Since each factor has a minuscule edge, \lem{l:smooth} tells us that
$Q(A(\vlambda,\vmu))$ is also a smooth variety.  Moreover, the dimensions
add to tell us that
$$\dim_\C Q(A(\vlambda,\vmu)) = \braket{\lambda_1 + \dots + \lambda_n,\rhov}
    = \dim_\C F(\vlambda).$$
On the other hand, $F(\vlambda) = Q(P(\vlambda))$ is partitioned as
a set by the subvarieties $Q(A(\vlambda,\vmu))$, simply by taking the
distances between the vertices of $P(\vlambda)$ and the origin.  If $X$
is any algebraic variety with an equidimensional partition into smooth
varieties $X_1,\ldots,X_N$, then $X$ has pure dimension and its components
are the closures of the parts $X_k$.  In our case, $X = F(\vlambda)$.
\end{proof}

It will be convenient later to abbreviate the dimension
of $F(\vlambda)$ as:
$$d(\vlambda) \defeq \braket{\lambda_1 + \dots + \lambda_n,\rhov}
    = \dim_\C F(\vlambda).$$
The same integers also arise in a different dimension formula:
$$\dim_\C \Gr(\vlambda) = 2d(\vlambda).$$
(Indeed, $\Gr(\vlambda)$ is a top-dimensional component of $F(\vlambda
\concat \vlambda^*)$, given by collapsing the polygon $P(\vlambda \concat
\vlambda^*)$ onto the polyline $L(\vlambda)$.)

Another important corollary of \lem{l:smooth} is the following:

\begin{theorem} Suppose that $D$ is a diskoid with boundary $\vlambda$
with no internal vertices, and suppose that all edges of $D$
(including the terms of $\vlambda$) are minuscule.  Then $Q(D)$
is smooth and projective, and therefore a single component
of $F(\vlambda)$.
\label{th:smooth} \end{theorem}

\begin{proof} Let $T_e(\mu,\lambda,\nu)$ be a triangle of $D$ with
three minuscule edges, and let the base edge $e$ be any of the edges.
Then by \lem{l:smooth}, $Q(T_e(\mu,\lambda,\nu))$ is smooth.  Likewise
$T_p(\mu,\lambda,\nu)$, based at a point $p$ instead, is smooth.
By construction, $Q(D)$ is a twisted product of configuration spaces of
this form, so it is also smooth.  It is also projective since $D$ is a
minuscule linkage.

There is one delicate point in the inference that $Q(D)$ is a component of
$F(\vlambda)$:  Is the restriction map $Q(D) \to F(\vlambda)$ injective?
As in the proof of \lem{l:smooth}, the restriction map
$$\pi:Q(T_e(\mu,\lambda,\nu)) \to \Gr(\lambda)$$
is injective, and so is the restriction map
$$\pi:Q(T(\mu,\lambda,\nu)) \to \Gr(\mu,\lambda).$$
The diskoid $D$ must have a triangle with at least two edges on the boundary,
so by induction its restriction map to $F(\vlambda)$ is also injective.
\end{proof}

\subsection{A homological state model}
\label{s:hstate}

This subsection discusses our motivation for the technical constructions
in the remainder of \sec{s:minuscule}.

We would like to use \thm{th:satake} as a state model or counting model
to evaluate webs in $\rep^u(G)$.  If $w$ is a web with dual diskoid $D$,
then there is a map of linkages
$$P(\vlambda) = \partial D \longto \Gamma(D)$$
given by the inclusion of the boundary.  This gives rise to a restriction map
$$\pi = \pi^{\Gamma(D)}_{P(\vlambda)}: Q(D) \to F(\vlambda).$$
A point in $Q(D)$ is a ``state" of $D$ in the sense of mathematical
physics, in which each vertex of $D$ (or each face of $w$) is assigned
an element of $\Gr$.  We would like to count the number of states of $D$
with some fixed boundary, or in other words the cardinality of a diskoid
fibre $\pi^{-1}(f)$ for $f \in F(\vlambda)$.  If $f$ is chosen generically
in a top-dimensional component of $F(\vlambda)$, then optimistically this
cardinality will be the coefficient of $\Psi(w)$ in the Satake basis.

However, this sketch is naive.  The diskoid fibre $\pi^{-1}(f)$ often
has a complicated geometry for which it is hard to define ``counting".
The first and main solution for us is to replace counting by a homological
intersection.  (In \sec{s:euler} we will propose a second solution, in which
we count by taking the Euler characteristic of the fibre.)  In particular,
for each web $w$, we will define a homology class $c(w) \in H_\top(Q(D))$
such that $\pi_*(c(w))$ equals $\Psi(w)$.

\subsection{The homology convolution category}
\label{s:conv}

If $M$ is an algebraic variety over $\C$, we will consider its
intersection cohomology sheaf $IC_M$ as a simple object in the category
of perverse sheaves on $M$.   If $M$ is smooth, then $IC_M$ is isomorphic
to $\C_M[\dim_\C M]$, the constant sheaf shifted by the complex dimension
of $M$. For brevity, we will write this perverse sheaf as $\C[M]$.

The geometric Satake correspondence is a tensor functor that takes the usual
product on $\rep^u(G)$ to the convolution tensor product on $\perv(\Gr)$.
In particular, the tensor product $V(\vlambda)$ of irreducible minuscule
representations corresponds to the convolution tensor product of the simple
perverse sheaves $\C[\Gr(\lambda_i)]$ on minuscule spheres, which are
closed in the algebraic geometry.  By definition, this convolution tensor
product is given by the pushforward $(m_{\vlambda})_*(\C[\Gr(\vlambda)])$
along the convolution morphism.

Let $\perv(\Gr)_\min$ denote the subpivotal category of $\perv(\Gr)$
consisting of such pushforwards.  By construction, $\perv(\Gr)_\min$
is equivalent to $\rep^u(G)_\min$.  Our goal is to study $\perv(\Gr)_\min$
using convolutions in homology, following ideas of Ginzburg.  We begin by
reviewing some generalities, following \cite[Sec. 2.7]{CG:complex}.

Let $\{M_i\}$ be a set of connected, smooth complex varieties and let $M_0$
be a possibly singular, stratified variety with strata $\{U_\alpha\}$.
For each $i$, let $\pi_i:M_i \to M_0$ be a proper semismall map.
In this context, the statement that $\pi_i$ is semismall means that
$\pi_i$ restricts to a fibre bundle over each stratum $U_\alpha$ and
that the dimensions of these fibres is given by
$$\dim_\C \pi_i^{-1}(u) = \frac{\dim_\C M_i - \dim_\C U_\alpha}2$$
for $u \in U_\alpha$ (note that we have equality above).  Let $d_i =
\dim_\C M_i$.

With this setup, let $Z_{ij} = M_i \times_{M_0} M_j$.  The semismallness
condition implies that $\dim_\C Z_{ij} = \frac{d_i + d_j}{2}$.  Let
$$H_\top(Z_{ij}) = H_{d_i + d_j}(Z_{ij})$$
be the top homology of $Z_{ij}$.  If the $M_i$ are proper, which they
will be in our situation, then we will obtain a valid definition of the
convolution product using the ordinary singular homology of $Z_{ij}$.
(Otherwise the correct type of homology would be Borel-Moore homology.)

Define a \emph{homological convolution product}
$$*:H_\top(Z_{ij}) \tensor H_\top(Z_{jk}) \to H_\top(Z_{ik})$$
by the formula
$$c_1 * c_2 = (\pi_{ik})_* (\pi_{ij}^*(c_1) \cap \pi_{jk}^*(c_2)),$$
where ``$\cap$" denotes the intersection product (with support), relative
to the ambient smooth manifold $M_i \times M_j \times M_k$.  This may
be defined using the cup product in cohomology via Poincar\'e duality.
For more details about this construction, see \cite[Sec. 2.6.15]{CG:complex}
or \cite[Sec. 19.2]{Fulton:intersect}.  Note that because
$$\dim_\C Z_{ij} = \frac{d_i + d_j}{2},$$
the correct homological degree is preserved by the convolution product.

This construction is relevant for us because of a theorem of Ginzburg
that relates $H_\top(Z_{ij})$ to morphisms in the category $\perv(M_0)$
of perverse sheaves on $M_0$.

\begin{theorem} \cite[Thm. 8.6.7]{CG:complex}
With the above setup, there is an isomorphism
$$H_\top(Z_{ij}) \cong
    \Hom_{\perv(M_0)} \bigl( (\pi_i)_*\C[M_i], (\pi_j)_*\C[M_j] \bigr).$$
This isomorphism identifies convolution products on the left side with
compositions of morphisms on the right side.
\label{th:ginzburg}
\end{theorem}

We will apply this setup by letting $M_0 = \Gr$ and by letting each $M_i$
be $\Gr(\vlambda)$ for a sequence $\vlambda$ of dominant minuscule weights.
The convolution morphism $m_\vlambda: \Gr(\vlambda) \to \Gr$ is semismall.
(See \cite[Lem. 4.4]{MV:geometric}; it also follows from the proof of
\thm{th:haines}.)  Then $Z_{ij}$ becomes
$$Z(\vlambda, \vmu) = \Gr(\vlambda) \times_\Gr \Gr(\vmu)
    = Q(P(\vlambda^* \concat \vmu)),$$
where $P(\vlambda^* \concat \vmu)$ is this polygon:
$$P(\vlambda^* \concat \vmu) \;\;=\;\; \begin{tikzpicture}[baseline]
\draw[darkred,semithick] (0,0) -- (1,1) -- (2,.5) -- (3,1) -- (3.5,0)
    -- (2,-1) -- (1,-1) -- cycle;
\draw[darkred,semithick,midto] (1,1) -- (2,.5);
\draw[darkred,semithick,midto] (1,-1) -- (2,-1);
\fill[darkred] (0,0) circle (.07);
\path[draw=darkred,fill=white] (3.5,0) circle (.07);
\draw[anchor=north] (1.75,-1.1) node {$\vlambda$};
\draw[anchor=south] (1.75,.9) node {$\vmu$};
\end{tikzpicture}$$

\thm{th:ginzburg} motivates the following construction of a category
$\hconv(\Gr)$.  The objects in $\hconv(\Gr)$ are the polyline varieties
$\Gr(\vlambda)$, where $\vlambda$ is a sequence minuscule weights.
The tensor product on objects is, by definition, given by convolution on
objects, so
$$\Gr(\vlambda) \tensor \Gr(\vmu) \defeq \Gr(\vlambda \concat \vmu),$$
where $\concat$ denotes concatenation of sequences.  So the identity object
is the point $\Gr(\emptyset)$.  Finally the dual object $\Gr(\vlambda)^*
= \Gr(\vlambda^*)$ of $\Gr(\vlambda)$ is given by reversing $\vlambda$
and taking the dual of each of its terms.

We define the morphism spaces of $\hconv(\Gr)$ as
$$\Hom_{\hconv(\Gr)}(\Gr(\vlambda), \Gr(\vmu))
    \defeq H_\top(Z(\vlambda, \vmu)).$$
The composition of morphisms is given by the convolution product.  Note that
the identity morphism $1_{\vlambda} \in H_\top(Z(\vlambda, \vlambda))$
is given by the class $[\Gr(\vlambda)_\Delta]$ of the diagonal
$$\Gr(\vlambda)_\Delta \subseteq Z(\vlambda,\vlambda) \subseteq
    \Gr(\vlambda) \times \Gr(\vlambda),$$
\ie, it is the configurations in which the polygon $P(\vlambda^* \concat
\vlambda)$ has collapsed onto the polyline $L(\vlambda)$.

To describe the tensor structure on morphisms, it is enough to describe
how to tensor with the identity morphism.  So let $\vlambda, \vmu, \vnu$ be
three sequences of dominant minuscule weights and let $c \in H_\top(Z(\vmu,
\vnu))$.  Our goal is to construct a class
$$1_{\vlambda} \tensor c \in
    H_\top(Z(\vlambda \concat \vmu, \vlambda \concat \vnu))$$

For the moment, let $\Gamma$ be a $\rho$-shaped graph with a tail of type
$\vlambda$ and a loop of type $\vmu^* \concat \vnu$, based at the end of
the tail:
$$\Gamma \;\;=\;\; \begin{tikzpicture}[baseline]
\draw[darkred,semithick] (0,0) -- (1,1) -- (2,.5) -- (3,1) -- (3.5,0)
    -- (2,-1) -- (1,-1) -- cycle;
\draw[darkred,semithick] (0,0) -- (-1,.5) -- (-2,0);
\draw[darkred,semithick,midto] (1,1) -- (2,.5);
\draw[darkred,semithick,midto] (1,-1) -- (2,-1);
\draw[darkred,semithick,midto] (-2,0) -- (-1,.5);
\fill[darkred] (-2,0) circle (.07);
\path[draw=darkred,fill=white] (0,0) circle (.07);
\path[draw=darkred,fill=white] (3.5,0) circle (.07);
\draw[anchor=north] (1.75,-1.1) node {$\vmu$};
\draw[anchor=south] (1.75,.9) node {$\vnu$};
\draw[anchor=north] (-1,.1) node {$\vlambda$};
\end{tikzpicture}$$
Let $X = Q(\Gamma)$ be its based configuration space.  We describe two
fibration constructions related to $X$.  First, there is a restriction map
$$\pi_{L(\vlambda) \sqcup \pt}^{L(\vlambda \concat \vmu)}:
    \Gr(\vlambda \concat \vmu) \to \Gr(\vlambda) \times \Gr$$
given by restricting to the polyline $L(\vlambda)$ and the free endpoint
of $L(\vlambda \concat \vmu)$.  Then $X$ is the fibred product
$$X = \Gr(\vlambda \concat \vmu)
    \times_{\Gr(\vlambda) \times \Gr} \Gr(\vlambda \concat \vnu).$$
Second, there is a projection
$$\pi_{L(\vlambda)}^\Gamma:X \to \Gr(\vlambda)$$
given by restricting from $\Gamma$ to $L(\vlambda)$.  The fibres
of this projection are $Z(\vmu,\vnu)$.

Since $\Gr(\vlambda)$ is simply connected, we get an isomorphism
$$H_\top(X) \cong H_\top(\Gr(\vlambda)) \tensor H_\top(Z(\vmu, \vnu))$$
and thus we obtain an isomorphism
$$H_\top(Z(\vmu, \vnu)) \stackrel{\cong}{\longto} H_\top(X)$$
given by $c \mapsto [\Gr(\vlambda)] \tensor c$.

There is also an inclusion
$$i = \pi_{P(\vlambda \concat \vmu \concat \vnu^* \concat \vlambda^*)}^\Gamma:
    X \to Z(\vlambda \concat \vmu, \vlambda \concat \vnu),$$
using the polygon which travels twice along the tail of $\Gamma$
and around the loop of $\Gamma$.  Combining all this structure, we define
$$1_{\vlambda} \tensor c \defeq i_*( [\Gr(\vlambda)] \tensor c).$$

Tensoring by the identity morphism on the other side is similar and we
leave the construction to the reader.

Finally, to define the cap and cup morphisms for any $\vlambda$, we will
define them for a single minuscule weight $\lambda$.  Note that
$$Z(\lambda \concat \lambda^*, \emptyset)
    = Z(\emptyset, \lambda \concat \lambda^*)
    = F(\lambda, \lambda^*) \cong \Gr(\lambda).$$
We define the cup $b_\lambda$ and the cap $d_\lambda$ to each be the class
$[\Gr(\lambda)]$ in their respective hom spaces.

\begin{theorem} There is an equivalence of pivotal categories
$$\hconv(\Gr) \cong \perv(\Gr)_\min \cong \rep^u(G)_\min.$$
\label{th:convperv} \end{theorem}

Applying \thm{th:convperv} to invariant spaces, we obtain an isomorphism
$$\Inv(V(\vlambda)) \cong \Hom_{\hconv(\Gr)}(\Gr(\emptyset),\Gr(\vlambda))
    = H_\top(Z(\emptyset, \vlambda)) = H_\top(F(\vlambda)),$$
which is \thm{th:satbasis}.

\begin{proof} The second equivalence is geometric Satake, so we will just
prove the first equivalence.  We begin by showing that it is an equivalence
of monoidal categories.

By the definition, the objects in both categories are parameterized
by sequences $\vlambda$, so the functor on objects is very simple.  On
morphisms, the functor is given by the isomorphisms from \thm{th:ginzburg}.
By this theorem, the functor is fully faithful and is compatible with
composition on both sides. (\Ie, it is a functor.)  To complete the proof
this theorem, we need only to show that the functor is compatible with
the tensor product and with pivotal duality.

To see that it is compatible with the tensor product, we use the same
notation as above.  If
$$c \in \Hom((m_\vmu)_* \C[\Gr(\vmu)], (m_\vnu)_*\C[\Gr(\vnu)]),$$
then with respect to the tensor structure in $\perv(\Gr)$,
$I_{(m_\vlambda)_*\C[\Gr(\vlambda)]} \tensor c$ is given by the image of
$c$ under the map
\begin{multline*}
\Hom_{\perv(\Gr)}\left( (m_\vmu)_* \C[\Gr(\vmu)],
    (m_\vnu)_* \C[\Gr(\vnu)] \right) \stackrel{\cong}{\longto} \\
\Hom_{\perv(\Gr(\vlambda) \times \Gr)} \left( (\pi_{L(\vlambda) \sqcup \pt}
    ^{L(\vlambda \concat \vmu)})_* \C[\Gr(\vlambda \concat \vmu)],
    (\pi_{L(\vlambda) \sqcup \pt}^{L(\vlambda \concat \vnu)})_*\C[\Gr(\vlambda
    \concat \vnu)] \right) \stackrel{p_*}{\longto} \\
\Hom_{\perv(\Gr)} \left( (\pi_{\vlambda \concat \vmu})_*
    \C[\Gr(\vlambda \concat \vmu)], (m_{\vlambda \concat \vnu})_*
    \C[\Gr(\vlambda \concat \vnu)] \right).
\end{multline*}
Here $p: \Gr(\vlambda) \times \Gr \to \Gr$ is the projection onto the
second factor.  This is easily seen to match our above definition.

It remains to check that the pivotal structures match under this equivalence.
Recall from \sec{s:signs} that the pivotal structures on $\rep(G)$ are
determined by the dimensions $\dim(V(\lambda))$, which are by definition
the values of closed loops.  (The discussion there is for pivotal structures
that differ by a sign, but it is true in general.)  Moreover, the discrepancy
is multiplicative, so it only needs to be checked for minuscule $\lambda$.

Let $\lambda$ be minuscule.  In $\hconv(\Gr)$, the value of a loop labelled
$\lambda$, \ie, the composition
$$d_\lambda \circ b_\lambda \in \Hom(\Gr(\emptyset), \Gr(\emptyset)) = \C,$$
is given by the self-intersection of $\Gr(\lambda) \cong F(\lambda,
\lambda^*)$ with itself inside $\Gr(\lambda, \lambda^*)$.

There is a neighbourhood (defined using the pullback of the open big cell)
of $F(\lambda, \lambda^*)$ in $\Gr(\lambda, \lambda^*)$ which is isomorphic
to $T^*\Gr(\lambda)$, under an isomorphism which carries $F(\lambda,
\lambda^*)$ to the zero section $\Gr(\lambda)$.

For any compact, complex $d$-manifold $X$, the self-intersection of $X$
with itself inside $T^*X$ is $(-1)^d \chi(X)$, where $\chi(X)$ is the Euler
characteristic of $X$.  (The self-intersection in $TX$ is $\chi(X)$, and
for a complex $d$-manifold the cotangent bundle $T^*X$ has the opposite real
orientation exactly when $d$ is odd.)  Applying this to $X = \Gr(\lambda)$,
we conclude that
$$d_\lambda \circ b_\lambda = (-1)^d \chi(\Gr(\lambda))
    = (-1)^{\braket{2\lambda, \rhov}} \dim V(\lambda).$$
This is the sign correction that is used to define the pivotal structure
on $\rep^u(G)$, as desired.
\end{proof}

\subsection{From the free spider to the convolution category}

\sec{s:spiders} describes a pivotal functor
$$\Psi:\fsp(G) \to \rep^u(G)_\min.$$
On the other hand, the geometric Satake correspondence and \thm{th:convperv}
yield equivalences
$$\rep^u(G)_\min \cong \perv(\Gr)_\min \cong \hconv(\Gr).$$
The composition is a functor $\fsp(G) \to \hconv(\Gr)$ which we will
also denote by $\Psi$.  Our goal now is to describe this functor and in
particular its action on invariant vectors.

Let $\lambda, \mu, \nu$ be a triple of dominant minuscule weights such that
$$\Inv_G(V(\lambda, \mu, \nu)) \ne 0.$$
There is a simple web $w \in
\Inv_{\fsp(G)}(\lambda, \mu, \nu)$ which contains a single vertex. On the
other hand,
$$\Inv_{\hconv(\Gr)}(\lambda, \mu, \nu) \cong H_\top(F(\lambda, \mu, \nu))$$
is one-dimensional with canonical generator $[F(\lambda, \mu, \nu)]$.
Recall from \sec{s:spiders} that in the construction of the functor
$\fsp(G) \to \rep^u(G)_\min$, there was some freedom to choose the image
of the simple web $w$ (it was only defined up to a non-zero scalar).  Now,
we fix this choice by setting
$$\Psi(w) \defeq [F(\lambda, \mu, \nu)].$$
The functor $\Psi$ is now determined by what it does on vertices and the
fact that it preserves the pivotal structure on both sides.

We are now in a position to prove \thm{th:homclass}, which we will restate
as follows.  Recall that
$$d(\vlambda) = \dim_\C F(\vlambda).$$

\begin{theorem} Let $w$ be a web with boundary $\vlambda$ and dual diskoid
$D = D(w)$.  Let
$$\pi:Q(D) \to F(\vlambda)$$
be the boundary restriction map.  There exists a homology class $c(w)
\in H_{2d(\vlambda)}(Q(D))$ such that $\pi_*(c(w)) = \Psi(w)$.  Moreover,
when $Q(D)$ has dimension $d(\vlambda)$ and is reduced as a scheme, then
$c(w)$ is the fundamental class $[Q(D)]$.
\label{th:homclass2} \end{theorem}

\begin{proof}
We begin by picking a isotopy representative for $w$ such that the height
function is a Morse function and so that the boundary of $w$ is at the
top level.  We assume a sequence of horizontal lines $\ell_0,\ldots,\ell_m$
such that in between each pair, $w$ has only a single cap, cup, or a vertex.
We assume further that each vertex is either an ascending Y (it is in the
shape of a Y) or a descending Y (an upside-down Y).

\begin{fullfigure}{f:morseweb}{A web for $\SL(9)$ in Morse position}
\begin{tikzpicture}[scale=0.5]
\begin{scope}[web]
\draw[midto] (0,10) -- (0,8) node[black,left,midway] {$\omega_2$};
\draw[midto] (0,8) -- (0,6);
\draw[midto] (0,6) -- (1,5); \draw[midfrom] (1,5) -- (1,4);
\draw[midfrom] (1,4) -- (1,2);
\draw[midfrom] (3,10) -- (3,9) node[black,left,midway] {$\omega_3$};
\draw[midto] (3,9) -- (2,8);
\draw[midto] (2,8) -- (2,6) node[black,left,midway] {$\omega_3$};
\draw[midto] (2,6) -- (1,5);
\draw[midto] (3,9) -- (4,8) node[black,right,midway] {$\omega_3$};
\draw[midto] (4,8) -- (5,7); \draw[midto] (5,7) -- (5,6);
\draw[midto] (5,6) -- (5,4) node[black,left,midway] {$\omega_8$};
\draw[midto] (5,4) -- (6,3);
\draw[midto] (6,3) -- (6,2);
\draw[midto] (6,10) -- (6,8) node[black,left,midway] {$\omega_5$};
\draw[midto] (6,8) -- (5,7);
\draw[midfrom] (7,10) -- (7,8) node[black,right,midway] {$\omega_4$};
\draw[midfrom] (7,8) -- (7,6); \draw[midfrom] (7,6) -- (7,4);
\draw[midfrom] (7,4) -- (6,3);
\draw[midfrom] (1,2) .. controls (1,0) and (6,0) .. (6,2)
    node[black,above,midway] {$\omega_4$};
\end{scope}
\foreach \k in {0,1,...,5}
    \draw[dashed] ($(-1,2*\k)$) -- ($(8,2*\k)$) node[right] {$\ell_\k$};
\end{tikzpicture}
\end{fullfigure}

Let $\vlambda^{(k)}$ be the vector of labels of the edges cut by
the horizontal line $\ell_k$.  Then $\vlambda^{(0)} = \emptyset$ and
$\vlambda^{(m)} = \vlambda$. For example, in \fig{f:morseweb} shows an
$\SL(9)$ web in Morse position, with edges labelled by its minuscule
weights $\omega_k$ with $1 \le k \le 8$.  In this example,
$$\vlambda^{(1)}=\{\omega_4,\omega_5\} \qquad
    \vlambda^{(3)}=\{\omega_7,\omega_6,\omega_1,\omega_4\}.$$
(Note that in $\SL(n)$ in general, $\omega_k^* = \omega_{n-k}$; if an edge
points down as it crosses a line, then we must take the dual weight.)

Let
$$w_k \in \Hom_{\fsp(G)}(\vlambda^{(k-1)}, \vlambda^{(k)})$$
denote the web in the horizontal strip between the lines $\ell_{k-1}$
and $\ell_k$.  By examining the above definition, we see that for each
$1 \le k \le m$, there exists a component $X_k \subset Z(\vlambda^{(k-1)},
\vlambda^{(k)})$ such that $\Psi(w_k) = [X_k]$.  We would like to describe
this component explicitly.  For convenience, if
$$\vp = (p_0,p_1,\ldots,p_m) \in \Gr^{m+1}$$
(with $p_0 = t^0$ for us), define $\sigma_i(\vp)$ by omitting the term $p_i$.

\begin{enumerate}
\item If $w_k$ is an ascending Y vertex that connects the $i$th point on
$\ell_{k-1}$ to the $i$th and $i+1$st points on $\ell_k$, then
$$X_k = \{(\vp,\vp') \in Z(\vlambda^{(k-1)}, \vlambda^{(k)})
    | \vp = \sigma_i(\vp')\}.$$

\item If $w_k$ is a descending Y vertex that connects the $i$th
and $i+1$st points on $\ell_{k-1}$ to the $i$th point on $\ell_k$, then
$$X_k = \{(\vp,\vp') \in Z(\vlambda^{(k-1)}, \vlambda^{(k)})
    | \vp' = \sigma_i(\vp)\}.$$

\item If $w_k$ is a cup that connects the $i$th and $i+1$st points on $\ell_k$,
then
$$X_k = \{(\vp,\vp') \in Z(\vlambda^{(k-1)}, \vlambda^{(k)})
    | \vp = \sigma_i(\sigma_i(\vp')) \}.$$

\item If $w_k$ is a cap that connects the $i$th and $i+1$st points on
$\ell_{k-1}$, then
$$X_k = \{(\vp,\vp') \in Z(\vlambda^{(k-1)}, \vlambda^{(k)})
    | \vp' = \sigma_i(\sigma_i(\vp)) \}.$$
\end{enumerate}

Then $w = w_m \circ \cdots \circ w_1$.   Since $\Psi$ is a functor,
$$\Psi(w) = \Psi(w_m) * \cdots * \Psi(w_1) = [X_m] * \cdots * [X_1].$$

Now, compositions of convolutions can be computed as a single convolution as
$$[X_m] * \cdots * [X_1]
    = (\pi_{0,m})_*(\pi_{0,1}^* [X_1] \cdots \pi_{m-1, m}^*[X_m]),$$
where the intersection products take place in the ambient smooth manifold
$$X = \Gr(\vlambda^{(0)}) \times \cdots \times \Gr(\vlambda^{(m)}).$$
Here $\pi_{k-1,k}$ denotes the projection from $X$ to
$\Gr(\vlambda^{(k-1)}, \vlambda^{(k)})$.

From the definitions, we see that the diskoid configuration spaces $Q(D)$
can be obtained as
$$Q(D) = \pi_{0,1}^{-1}(X_1) \cap \cdots \cap \pi_{m-1, m}^{-1}(X_m).$$
Let
\begin{align*}
c(w) &= \pi_{0,1}^* [X_1] \cap \cdots \cap \pi_{m-1, m}^*[X_m] \\
    &= [\pi_{0,1}^{-1}(X_1)] \cap \cdots \cap [\pi_{m-1, m}^{-1}(X_m)].
\end{align*}
Because we are using the intersection product with support, $c(w)$ lives
in $H_{d(\vlambda)}(Q(D)) $, the homology of the intersection.  When $Q(D)$
is reduced of the expected dimension, then the intersection product of the
homology classes corresponds to the fundamental class of the intersection
(see \cite[Sec. 8.2]{Fulton:intersect}), so $c(w) = [Q(D)]$.

Finally, $\pi: Q(D) \to F(\vlambda)$ is the restriction of $\pi_{0,m}$
to $Q(D)$.  Hence we conclude that $\Psi(w) = \pi_*(c(w))$.
\end{proof}

Because $\pi_*(c(w))$ is supported on $\pi(Q(D))$, we immediately obtain
the following.

\begin{corollary} $\Psi(w)$ is a linear combination of the fundamental
classes of the components of $F(\vlambda)$ which are in the image of $\pi$.
\label{c:support} \end{corollary}

It may not seem clear that $c(w)$ depends only on the web $w$, and not on
the Morse position of $w$ used to construct it.  However, a posteriori,
this must be verified by checking that it is invariant under basic isotopy
moves (for example, straightening out a cup/cap pair).

\section{$\SL(3)$ results}
\label{s:sl3}

In this section, we will prove \thm{th:main} and \thm{th:unitriang}.
In preparation for this result, we need to use and extend the geometry of
non-elliptic webs.  To review, if $w$ is an $A_2$ web and $D = D(w)$ is
its dual diskoid, then $w$ is non-elliptic if and only if $D$ is $\CAT(0)$.

\subsection{Geodesics in $\CAT(0)$ diskoids}
\label{s:geodesics}

We will be interested in combinatorial (meaning edge-travelling) geodesics
in a type $A_2$ diskoid $D$.  These are equivalent to ``minimal cut paths"
of the dual web \cite{Kuperberg:spiders}, when the endpoints of the geodesic
are boundary vertices $D$.  Here we will consider geodesics between vertices
that may be in the interior or on the boundary.  If both vertices are on
the boundary, then the geodesic is called \emph{complete}.

\begin{fullfigure}{f:diamond}{Two geodesics $\gamma$ and $\gamma'$
    connected by a diamond move}
\begin{tikzpicture}[x={(.75cm,0)},y={(.375cm,.659cm)}]
\draw[fill=lightgray] (0,0) -- (1,0) -- (1,1) -- (0,1) -- cycle;
\draw (1,0) -- (0,1); \draw (-2,0) -- (0,0);
\draw (1,1) -- (2,1) -- (2,2);
\draw[darkred,semithick] (-2,.15) -- (-.15,.15)
    -- (-.15,1.15) -- (1.85,1.15) -- (1.85,2);
\draw[darkred,semithick] (-2,-.15) -- (1.15,-.15)
    -- (1.15,.85) -- (2.15,.85) -- (2.15,2);
\node[anchor=south east] at (-.15,1.15) {$\gamma$};
\node[anchor=north west] at (1.15,-.15) {$\gamma'$};
\end{tikzpicture}
\end{fullfigure}

Geodesics in an $A_2$ diskoid are often not unique.  Define a \emph{diamond
move} of a geodesic to be a move in which the geodesic crosses two triangles,
as in \fig{f:diamond}.  (This is equivalent to an ``$H$-move" on a cut
path of a non-elliptic web.)  We say that two geodesics are \emph{isotopic}
if they are equivalent with respect to diamond moves.

\begin{theorem}  Let $p,q$ be two vertices of a $\CAT(0)$, type $A_2$
diskoid $D$.  Then the geodesics between $p$ and $q$ subtend a diskoid
which is a skew Young diagram, with each square split into two triangles.
In particular, all geodesics are isotopic, $D$ is geodesically coherent,
and all geodesics lie between two extremal geodesics.  Both of the extremal
geodesics are concave on the outside.
\label{th:young} \end{theorem}

Here a \emph{skew Young diagram} is the same as the usual object in
combinatorics with that name, namely the diskoid lying between two geodesic
lattice paths in $\mathbb{Z}^2$.  \fig{f:skew} shows an example in which
the squares have been split so that it becomes an $A_2$ diskoid.

\begin{fullfigure}{f:skew}{A skew partition bounded by extremal geodesics
    $\gamma$ and $\gamma'$}
\begin{tikzpicture}
\begin{scope}[x={(.75cm,0)},y={(.375cm,.659cm)}]
\foreach \x/\y in {0/0,1/0,2/1,3/1,2/2,3/2,3/3,5/4}
{
    \draw[fill=lightgray] (\x,\y) -- ++(1,0) -- ++(0,1) -- ++(-1,0) -- cycle;
    \draw (\x,\y) ++(1,0) -- ++(-1,1);
}
\draw (4,4) -- (5,4);
\draw[darkred,semithick] (-.15,-.15) -- ++(0,1.3) -- ++(2,0) -- ++(0,2)
    -- ++(1,0) -- ++(0,1) -- ++(2,0) -- ++(0,1) -- ++(1.3,0);
\draw[darkred,semithick] (-.15,-.15) -- ++(2.3,0) -- ++(0,1) -- ++(2,0)
    -- ++(0,3) -- ++(2,0) -- ++(0,1.3);
\coordinate (p) at (-.15,-.15); \coordinate (q) at (6.15,5.15);
\coordinate (g) at (1.85,3.15); \coordinate (g1) at (4.15,2.15);
\end{scope}
\fill[darkred] (p) circle (.07); \fill[darkred] (q) circle (.07);
\node[anchor=south east] at (p) {$p$};
\node[anchor=north west] at (q) {$q$};
\node[anchor=south east] at (g) {$\gamma$};
\node[anchor=north west] at (g1) {$\gamma'$};
\end{tikzpicture}
\end{fullfigure}

\thm{th:young} is proven in \cite{Kuperberg:spiders} in the case when $p$
and $q$ are on the boundary.  If they are not on the boundary, then we can
reduce to previous case by the removing the simplices of $D$ that do not
lie between two geodesics.  The final statement, that an extremal geodesic
$\gamma$ is concave outside of the skew Young diagram, is easy to check:
If $\gamma$ has an angle of $\pi/3$, then it is not a geodesic.  If it has
an angle of $2\pi/3$, then an isotopy is available and it is not extremal.

\begin{lemma} If $p$ and $q$ are two vertices of a $\CAT(0)$ diskoid $D$,
then every geodesic $\kappa$ between them extends to a complete geodesic.
\label{l:complete} \end{lemma}

\begin{proof} The argument is based on a geodesic sweep-out construction.
We claim that we can make a sequence of geodesics
$$\vgamma = (\gamma_0,\gamma_1,\dots,\gamma_{m-1})$$
from $p$ to the boundary $\partial D$ with certain additional properties.
We require that each consecutive pair $\gamma_k$ and $\gamma_{k+1}$ differ
by either an elementary isotopy or an elementary boundary isotopy (for each
$k \in \Z/m$).  The latter consists either of appending an edge to
$\gamma_k$ or removing the last edge, or a \emph{triangle move} as in
\fig{f:triangle}.

\begin{fullfigure}{f:triangle}{A triangle move connecting
    geodesics $\gamma_k$ and $\gamma_{k+1}$}
\begin{tikzpicture}
\begin{scope}[x={(.75cm,0)},y={(.375cm,.659cm)}]
\draw[fill=lightgray] (-1,0) -- (-1,-1) -- (0,-1) -- cycle;
\draw (-1,0) -- (0,-1);
\draw (-1,-1) -- (-2,-1) -- (-2,-2);
\draw[darkred,semithick] (.15,-1.15) -- (-1.85,-1.15) -- (-1.85,-2);
\draw[darkred,semithick] (-1.15,.15) -- (-1.15,-.85)
    -- (-2.15,-.85) -- (-2.15,-2);
\coordinate (r1) at (.15,-1.15);
\coordinate (r2) at (-1.15,.15);
\coordinate (g1) at (-1,-1.15);
\coordinate (g2) at (-1.15,-.85);
\end{scope}
\fill[darkred] (r1) circle (.07);
\fill[darkred] (r2) circle (.07);
\node[anchor=north west] at (r1) {$r_k$};
\node[anchor=south east] at (r2) {$r_{k+1}$};
\node[anchor=north] at (g1) {$\gamma_k$};
\node[anchor=south east] at (g2) {$\gamma_{k+1}$};
\end{tikzpicture}
\end{fullfigure}

We require that the other endpoint $r_k$ of $\gamma_k$ travel all the way
around $\partial D$ in the counterclockwise direction, as in \fig{f:sweep}.

\begin{fullfigure}{f:sweep}{Making a sequence of geodesics that
    sweep out $D$}
\begin{tikzpicture}
\begin{scope}[x={(.75cm,0)},y={(.375cm,.659cm)}]
\draw[darkred,semithick] (0,0) -- ++(0,1) -- ++(-1,1) -- ++(0,1);
\draw (1,3) -- ++(-4,0) -- ++(0,-2) -- ++(3,-3) -- ++(2,0) -- ++(0,1)
    -- ++(1,-1) -- ++(0,2) -- ++(-1,1) -- ++(0,1) -- cycle;
\coordinate (r) at (-1,3);
\coordinate (g) at (0,1);
\coordinate (d) at (2,0);
\end{scope}
\fill[darkred] (0,0) circle (.07);
\fill[darkred] (r) circle (.07);
\node[anchor=west] at (0,-.1) {$p$};
\node[anchor=north west] at (r) {$r_k$};
\node[anchor=west] at (g) {$\gamma_k$};
\node at (d) {$D$};
\draw[thick,->] (100:.6) arc (100:300:.6);
\end{tikzpicture}
\end{fullfigure}

If $p$ is on the boundary, then $r_0 = p$, but this is okay.  It is easy to
see that if $\vgamma$ exists, then it uses every vertex in $D$.  There is
thus a geodesic $\gamma$ from $p$ to $r \in \partial D$ that contains $q$.
We can then repeat the argument with $r$ replacing $p$, to obtain a
geodesic $\gamma'$ from $r$ to some $s \in \partial D$ that contains $p$.
The geodesic $\gamma'$ might not contain $q$, much less all of $\kappa$.
However, because $D$ is geodesically coherent, the path
$$\gamma'' = \gamma'(s,p) \concat \kappa \concat \gamma(q,r)$$
is a geodesic and satisfies the lemma, as in \fig{f:replace}.

\begin{fullfigure}{f:replace}{A geodesic replacement argument}
\begin{tikzpicture}
\begin{scope}[x={(.75cm,0)},y={(.375cm,.659cm)}]
\draw[semithick,darkred] (0,0) -- ++(-2,2) -- ++(-1,0);
\draw[semithick,darkred] (0,0) -- ++(-1,0) -- ++(-2,2) -- ++(-1,0)
    -- ++(0,-1) -- ++ (-1,0);
\draw[semithick,darkred] (-5,1) -- ++(1,-1) -- ++(2,0) -- ++(1,-1) -- ++(1,0)
    -- ++(0,1) -- ++(2,0);
\draw[semithick,darkgreen] (-5.2,1.2) -- ++(1,0) -- ++(0,1) -- ++(2.2,0)
    -- ++(2,-2) -- ++(2,0);
\draw (-6,2) -- (-5,1) -- (-5,0);
\draw (2,-1) -- (2,1);
\coordinate (p) at (0,0); \coordinate (q) at (-3,2);
\coordinate (r) at (-5,1); \coordinate (s) at (2,0);
\coordinate (k) at (-1.5,1.5); \coordinate (g) at (-2,1);
\coordinate (g1) at (-3,0); \coordinate (g2) at  (-1,1.2);
\end{scope}
\fill[darkred] (p) circle (.07); \fill[darkred] (q) circle (.07);
\fill[darkred] (r) circle (.07); \fill[darkred] (s) circle (.07);
\node[anchor=north west] at (p) {$p$};
\node[anchor=north east] at (q) {$q$};
\node[anchor=east] at (r) {$r$\;};
\node[anchor=west] at (s) {\;$s$};
\node[anchor=north east] at (k) {$\kappa$};
\node[anchor=north east] at (g) {$\gamma$};
\node[anchor=north] at (g1) {$\gamma'$};
\node[anchor=south west] at (g2) {$\gamma''$};
\end{tikzpicture}
\end{fullfigure}

To prove the claim, let $\gamma_0$ be the geodesic of length $0$ if $p \in
\partial D$, and otherwise let $\gamma_0$ be the geodesic from $p$ to any
$r_0 \in \partial D$ which is counterclockwise extremal.  We construct
$\vgamma$ iteratively.  Given $\gamma_k$, we apply a diamond move to
make $\gamma_{k+1}$ if such a move is possible.  If such a move is not
possible, then let $r_{k+1}$ be the next boundary vertex after $r_k$,
and let $\gamma_{k+1}$ be the clockwise-extremal geodesic from $p$
to $r_{k+1}$, among geodesics that do not cross $\gamma_k$.  (In other
words, cut $D$ along $\gamma_k$ to make $D'$, then let $\gamma_{k+1}$
be clockwise-extremal in $D'$.)  By geodesic coherence, the region between
$\gamma_k$ and $\gamma_{k+1}$ is either empty or connected; otherwise
we could splice $\gamma_{k+1}$ with $\gamma_k$, so that $\gamma_{k+1}$ would
not be clockwise-extremal.

If the region between $\gamma_k$ and $\gamma_{k+1}$ is empty, then either
$\gamma_k \subseteq \gamma_{k+1}$ or $\gamma_{k+1} \subseteq \gamma_k$.
If it is not empty, then there are two geodesic segments $\gamma_k(s,r_k)$
and $\gamma_{k+1}(s,r_{k+1})$ make a topological triangle $T$ together
with the edge $(r_k,r_{k+1})$, as in \fig{f:toptriangle}.

\begin{fullfigure}{f:toptriangle}{A topological triangle $T$ made
    from geodesics}
\begin{tikzpicture}
\draw (30:1) -- (150:1);
\draw[darkred,semithick] (150:1) -- (210:.3)
    -- (270:1.5) -- (330:.3) -- (30:1);
\fill[darkred] (30:1) circle (.07); \fill[darkred] (150:1) circle (.07);
\fill[darkred] (270:1.5) circle (.07);
\node[above right] at (30:1) {$r_k$};
\node[above left] at (150:1) {$r_{k+1}$};
\node[below=.5ex] at (270:1.5) {$s$};
\node[below left] at (210:.3) {$\gamma_{k+1}$};
\node[below right] at (330:.3) {$\gamma_k$};
\node at (0,.1) {$T$};
\end{tikzpicture}
\end{fullfigure}

We summarize the properties of the topological triangle $T$: It is
$\CAT(0)$, all three sides are concave, and its angles at the corners are
at least $\pi/3$.  Thus $T$ is flat, all three sides are flat (unlike in
the figure), and all three angles equal $\pi/3$.  Thus, $T$ is a face of
$D$ and $\gamma_k$ and $\gamma_{k+1}$ differ by a triangle move.

As $k$ increases, eventually $r_k = r_0$.  Once the diamond moves are
exhausted for this choice of $r_k$ (there are none if $p$ is on the
boundary), the sequence of geodesics returns to the beginning.
\end{proof}

The sweep-out construction in the proof of \lem{l:complete} also
yields this lemma.

\begin{lemma} Let $D$ be a $\CAT(0)$ diskoid with a boundary vertex $p$.
Then every edge of $D$ either lies on a complete geodesic from $p$ to
some $q \in \partial D$, or it lies in a diamond move or a triangle move
between two geodesics from $p$.
\label{l:everyedge} \end{lemma}

Finally, there is a relation between fans as described in \sec{s:paths}
and non-elliptic webs.  Given a diskoid $D$ with boundary $\vlambda$, let
$\vmu(D)$ be the sequence of distances $d(p,q_k)$, where $p$ is the base
point of $D$ and $q_k$ is the sequence of boundary vertices of $D$.  Then:

\begin{theorem} \cite{Kuperberg:spiders} Given a sequence of $A_2$
minuscule weights $\vlambda$, the map $D \mapsto \vmu(D)$ is a bijection
between $\CAT(0)$ diskoids and minuscule paths of type $\vlambda$.
\end{theorem}

So we can write $D(\vlambda,\vmu)$ as the non-elliptic web with boundary
$\vlambda$ and minuscule path $\vmu$.

\subsection{Unitriangularity}
\label{s:unitriang}

We apply \sec{s:geodesics} to prove the following result.  It is a bridge
result, based on the geometry of affine buildings, that we will use to
relate web bases to the geometric Satake correspondence; in particular,
to prove \thm{th:unitriang}.

\begin{theorem} Let $\vlambda$ be a minuscule sequence of type $A_2$
and let $\vmu$ be a minuscule path of type $\vlambda$.  If $f \in
Q(A(\vlambda,\vmu))$ is a fan configuration, then it extends
uniquely to a diskoid configuration $f \in Q(D(\vlambda,\vmu))$.
\label{th:extend} \end{theorem}

\begin{proof} The construction derives from the constraints that make
the extension unique.  Let $p$ be the base vertex of $D$, so that $f(p)
= 0 \in \Gr$.   Suppose that $q$ is the $k$th boundary vertex of $D$, and
that $\gamma$ is a geodesic from $p$ to $q$.  Then $d(f(p),f(q)) = \mu_k$,
and by definition $\mu_k$ is also the length of $\gamma$.  If $\Sigma$ is
an apartment containing $f(p)$ and $f(q)$, then $f(q) = \mu_k$ in suitable
coordinates in $\Sigma$.  It follows that there is a unique geodesic in
$\Sigma$ with the same sequence of edge weights as $\gamma$, and which
connects $f(p)$ with $f(q)$.  Thus $f$ extends uniquely to $\gamma$.

We claim that this extension of $f$ is consistent for vertices of $D$. First,
every vertex of $D$ is contained in some complete geodesic from $p$ since
by \lem{l:complete} any geodesic from $p$ to a vertex extends to a
complete geodesic.  Suppose that $\gamma$ and $\gamma'$ are two geodesics
from $p$ to $q \in \partial D$ and $q' \in \partial D$, respectively.
Suppose further that $r \in \gamma \cap \gamma'$.  Then every apartment
that contains $p$ and $r$ contains both geodesics $\gamma(p,r)$ and
$\gamma'(p,r)$.  In particular, each apartment $\Sigma \supseteq \gamma$
and $\Sigma' \supseteq \gamma'$ does.  It follows that the choices for
$f(r)$ induced by $\gamma$ and $\gamma'$ are the same.

We claim that if $(r,s)$ is an edge in $D$, then
\eq{e:edge}{d(r,s) = d(f(r),f(s)).}
By \lem{l:everyedge}, there are three cases:  Either $(r,s)$ occurs
in a complete geodesic from $p$ to some $q$, or it occurs in a diamond
move between two such geodesics $\gamma$ and $\gamma'$, or $r$ and $s$
are both on the boundary and $(r,s)$ occurs in a triangle move between
two geodesics $\gamma$ and $\gamma'$.  In the first case, \eqref{e:edge}
is true by construction.  In the second case, $f(\gamma)$ and $f(\gamma')$
are contained in a single apartment, because every apartment contains all
geodesics from $f(p)$ to $f(q)$.  In the third case, there is an apartment
containing $p$ and $(r,s)$ by the axioms for a building, since they are
both simplices.  In both cases, the existence of this common apartment
implies \eqref{e:edge}.
\end{proof}

Now let $\vlambda$ be a minuscule dominant sequence, and let $\vmu$ be
a minuscule path of type $\vlambda$.  Then there is a corresponding
non-elliptic web $w(\vlambda, \vmu)$ with dual diskoid $D(\vlambda, \vmu)$.
There is also a corresponding component $\bar{Q(A(\vlambda, \vmu))}$
of $F(\vlambda)$.

We have two bases for $H_\top(F(\vlambda))$, one given by
$[\bar{Q(A(\vlambda, \vmu))}]$ and the other given by $\Psi(w(\vlambda,
\vmu))$, and both bases are indexed by the minuscule path $\vmu$.
Under the isomorphism
$$H_\top(F(\vlambda)) \cong \Inv(V(\vlambda)),$$
these become the Satake and web bases, respectively, the first by definition
and the second by \thm{th:homclass2}.  Our purpose in this section is to
prove that the transition matrix between these two basis is unitriangular.
Define a partial order on minuscule paths by the rule that
$\vnu\leq\vmu$ when $\nu_i\leq \mu_i$ for all $i$.

\begin{theorem} The transition matrix between the Satake and web bases is
unitriangular with respect to the partial order $\leq$.
\label{th:unitriangweak} \end{theorem}

In the next section, we will use this result to deduce \thm{th:unitriang},
which concerns a weaker partial order and is thus a stronger statement.

We divide the proof of \thm{th:unitriangweak} into the following two lemmas.

\begin{lemma} Suppose that $\vnu \not\le \vmu$.  Then the coefficient of
$[\bar{Q(A(\vlambda, \vnu))}]$ in $\Psi(w(\vlambda, \vmu))$ is 0.
\label{l:coeff2} \end{lemma}

\begin{proof}
By \cor{c:support}, it suffices to show that if $Q(A(\vlambda,
\vnu))$ is contained in $\pi(Q(D(\vlambda, \vmu)))$, then $\vnu \le \vmu$.

Let $f\in Q(D(\vlambda, \vmu))$.  If $q_i$ is the $i$th boundary vertex of
the diskoid $D(\vlambda, \vmu)$, then  $f(q_i) \in \bar{\Gr(\mu_i)}$.
On the other hand, if $\pi(f) \in Q(A(\vlambda, \vnu))$, then $f(q_i)
\in \Gr(\nu_i)$ . Thus $\nu_i \le \mu_i$ for all $i$ as desired.
\end{proof}

\begin{lemma} The coefficient of $[\bar{Q(A(\vlambda, \vmu))}]$
in $\Psi(w(\vlambda, \vmu))$ is 1.
\label{l:coeff1} \end{lemma}

\begin{proof}
Let $Z = \bar{\pi^{-1}(Q(A(\vlambda, \vmu))}$.  Then $Z$ is a component
of $Q(D(\vlambda, \vmu))$, and it has dimension $d(\vlambda)$ by
\thm{th:extend}.  Recall that from \thm{th:homclass2}, that we have a
homology class $c(w) \in H_{d(\vlambda)}(Q(D))$ such that $\pi_*(c(w))
= \Psi(w)$.  Using the notation of the proof of \thm{th:homclass2},
$$Q(D(\vlambda, \vmu)) =
    \pi_{0,1}^{-1}(X_1) \cap \cdots \cap \pi_{n-1, n}^{-1}(X_n)$$
and
$$c(w) = [\pi_{0,1}^{-1}(X_1)] \cap \cdots \cap [\pi_{n-1,n}^{-1}(X_n)]$$
Since $Z$ is a component of the expected dimension, we see that
the coefficient of $[Z]$ in $c(w)$ is the length of the local ring
of $Q(D(\vlambda, \vmu))$ along $Z$ (by \cite{Fulton:intersect},
Proposition 8.2).  This length equals 1 since following lemma shows that
the scheme $\pi^{-1}(Q(A(\vlambda, \vmu)))$ is isomorphic to the reduced
scheme $Q(A(\vlambda, \vmu))$.

The degree $\pi|_Z$ is 1, so $\pi_*([Z]) = [\bar{Q(A(\vlambda,
\vmu))}]$.  Moreover, $Z$ is the only component of $Q(D(\vlambda, \vmu))$
which maps onto $\bar{Q(A(\vlambda, \vmu))}$, so we conclude that the
coefficient of $[\bar{Q(A(\vlambda, \vmu))}]$ in $\pi_*c(w)$ is also 1,
as desired.
\end{proof}

\begin{lemma}
The restriction of the map $\pi: Q(D(\vlambda, \vmu)) \to F(\vlambda)$
to $\pi^{-1}(Q(A(\vlambda, \vmu))$ is an isomorphism of schemes onto the
reduced scheme $Q(A(\vlambda, \vmu))$.
\end{lemma}

\begin{proof}
First note that $Q(A(\vlambda, \vmu))$ is reduced since it is isomorphic to a
iterated fibred product of varieties by the proof of \thm{th:haines}.

Let $X = \pi^{-1}(Q(A(\vlambda, \vmu)), Y = Q(A(\vlambda, \vmu))$.  We have
already shown in \thm{th:extend} that the map $\pi: X \to
Y$ gives a bijection at the $\C$-points.  Now, let $S$ be any scheme of
finite-type over $\C$.  The proof of \thm{th:extend} uses some
building-theoretic arguments which don't obviously work for $S$-points.
However, the argument in the first paragraph of the proof does work
for any $S$, as follows.  Following the notation in that paragraph, let
$\gamma$ be a geodesic in $\Gamma$ from the base point $p$ of $D(\vlambda,
\vmu)$ to the $k$-th boundary vertex $q$ and let $\vnu$ be the lengths
along this geodesic (by definition $\sum \nu_i = \mu_k$).  Let $f \in
X(S)$. Then the restriction of the map $m: \Gr(\vnu) \to \Gr$
to $m^{-1}(\Gr(\mu_k))$ is an isomorphism of schemes, and in particular
is an injection on $S$-points.  Hence we see that $f(r)$ is determined by
$f(q)$ for all $r$ along the geodesic.  Since every internal vertex of the
diskoid lies on some geodesic, $f \in X(S)$ is determined by its restriction
to the boundary.  Thus, the map $X(S) \to Y(S)$ is injective.

So we have a map from a scheme to a smooth variety which is a bijection
on $\C$-points and is an injection on $S$-points.  By the following lemma,
the map is an isomorphism.
\end{proof}

\begin{lemma} Let $X, Y$ be finite-type schemes over $\C$. Assume that $Y$
is reduced and normal.  Let $\phi: X \to Y$ be a morphism which
induces a bijection on $\C$-points and an injection on $S$-points for all
finite-type $\C$-schemes $S$. Then $\phi$ is an isomorphism.
\end{lemma}

\begin{proof} Consider the maps
$$X_\mathrm{red} \to X \to Y.$$
The composition $X_\mathrm{red} \to Y$ is a bijection on $\C$-points and
hence it is an isomorphism \cite[Thm. A.11]{Kumar:kacmoody}.  This allows
us to construct a map $\psi:Y \to X$ such that $\phi \psi = id_Y$.

The fact that $\phi$ induces an injection on $S$-points means that the map
$$\Hom_{\Sch}(X, X) \xrightarrow{\phi \circ} \Hom_{\Sch}(X, Y)$$
is injective.  Consider what happens to $id_X$ and $\psi\phi$ under
this map.  They are sent to $\phi$ and $\phi\psi\phi$ respectively.
But since $\phi\psi = \mathrm{id}_Y$, these two elements of $\Hom_{\Sch}(X,
Y)$ are equal.  Hence by the injectivity, $\mathrm{id}_X = \psi \phi$
and hence $\phi$ is an isomorphism.
\end{proof}

\subsection{Consequences of the cyclic action}
\label{s:cyclic}

The goal of this section is to prove \thm{th:main} and then derive some
corollaries.  The proof is based on \thm{th:extend}.  However, we first
need to understand the cyclic action on webs and Satake fibres, \ie, the
action that results from changing the base point of a polygon or a diskoid.

Fix a minuscule sequence $\vlambda =(\lambda_1, \dots, \lambda_n)$ and
consider the corresponding Satake fibre $F(\vlambda)$.  Also regard the
indices of the sequence $\vlambda$ as lying in $\Z/n$.   For each $i \in
\Z/n$, we define
$$\vlambda^{(i)} = (\lambda_{i+1}, \lambda_{i+2}, \dots, \lambda_{n-1},
    \lambda_0, \lambda_1, \dots, \lambda_i)$$
to be the $i$th cyclic permutation of $\vlambda$ (so that $\vlambda^{(0)}
= \vlambda$).

Let $Z$ be an irreducible component of $F(\vlambda)$.  Since $G(\cO)$ is
connected and acts on $F(\vlambda)$, we see that $Z$ is $G(\cO)$-invariant.
Define
$$Z_1 = \{([g_1^{-1} g_2], \dots, [g_1^{-1} g_n], t^0)
    | ([g_1], \dots, [g_n]) \in Z \subset F(\vlambda^{(1)}) \},$$
and by iteration define $Z_i \subseteq F(\vlambda^{(i)})$ for all $i \in \Z/n$.
This yields a bijection
$$\Irr(F(\vlambda)) \cong \Irr(F(\vlambda^{(1)}))$$
which we call geometric rotation of components.
Another way to think about $Z_i$ is to think about the unbased configuration
space of $P(\vlambda)$ and to note that it fibres over $\Gr$ in $n$
different ways, by choosing each of the $n$ vertices of $P(\vlambda)$
as the base point.  (But the geometry of these fibrations is subtle,
because the fibres do not have to be isomorphic algebraic varieties.)

A straightforward calculation in convolution algebras (in which all
intersections are transverse), shows that geometric rotation matches
pivotal rotation in $\hconv(\Gr)$.  At the same time, \thm{th:satbasis}
tells us that the diagram
\begin{equation} \label{e:rotation}
\begin{tikzpicture}[description/.style={fill=white,inner sep=2pt},baseline]
\matrix (m) [matrix of math nodes, row sep=3em,
column sep=2.5em, text height=1.5ex, text depth=0.25ex]
{ H_\top(F(\vlambda)) & \Inv(V(\vlambda)) \\
H_\top(F(\vlambda^{(1)})) & \Inv(V(\vlambda^{(1)})) \\ };
\path[->] (m-1-1) edge node[auto] {$\Phi$} (m-1-2)
    (m-1-1) edge node[auto] {$R$} (m-2-1)
    (m-2-1) edge node[auto] {$\Phi$} (m-2-2)
    (m-1-2) edge node[auto] {$R$} (m-2-2);
\end{tikzpicture}
\end{equation}
commutes, where the invariant spaces on the right are in $\rep(G)^u_\min$.

By \thm{th:haines}, $Z = \bar{Q(A(\vlambda,\vmu))}$ for some minuscule path
$\vmu$ of type $\vlambda$.  From $(\vlambda,\vmu)$, we obtain a diskoid $D
=  D(\vlambda, \vmu)$.  In $D$, the distances from the base point to the
other boundary vertices are given by $\vmu$.  Now for each $i \in \Z/n$,
let $\vmu^{(i)}$ denote the sequence of distances from the $i$th boundary
vertex to the rest of the boundary.  Since a rotated $\CAT(0)$ diskoid is
still a $\CAT(0)$ diskoid, we see that $D = D(\vlambda^{(i)}, \vmu^{(i)})$
as well.

\begin{lemma}
For each $i$, $Z_i = \bar{Q(A(\vlambda^{(i)}, \vmu^{(i)}))}$.
\label{l:cyclic} \end{lemma}

Although this lemma may look purely formal, it is (as far as we know)
a non-trivial identification of two different cyclic actions.  The cyclic
action used to define $Z_i$ is defined directly from the geometric Satake
correspondence; it comes from the fact that the unbased configuration
space of $P(\vlambda)$ fibres over $\Gr$ in more than one way.  The cyclic
action on the right, in particular the definition of $\vmu^{(i)}$, comes
instead from rotating webs.  The two cyclic actions ``should be" the same
because the diagram analogous to \eqref{e:rotation} for webs commutes
(since $\spd(\SL(3))$ is equivalent to $\rep^u(\SL(3))$).  However, the
lemma is non-trivial because it is not true that the invariant vector
$\Psi(w(\vlambda, \vmu))$ coming from the web equals the fundamental class
of the corresponding component.

\begin{proof} Our proof uses \thm{th:unitriangweak}, the unitriangularity
theorem.  Let $M$ be the unitriangular change of basis matrix; the rows
of $M$ are labelled by the web basis, while the columns are indexed by the
geometric Satake basis.  Since both bases are cyclically invariant as in
the diagram \eqref{e:rotation}, there is a combinatorial cyclic action on
the rows and columns of $M$ that takes $M$ to itself.

Suppose for the moment that $M$ is an abstractly unitriangular matrix whose
rows and columns are labelled by two sets $A$ and $B$.  In other words,
there exists an unspecified bijection $A \cong B$, and a linear or partial
order of $A$ that makes $M$ unitriangular.  Then the partial order may
not be unique, but the bijection is.  If we choose any compatible linear
order, then it is easy to see that the expansion of $\det M$ has only
one non-zero term.  This term selects the unique compatible bijection.
Since it is unique, it intertwines the two cyclic actions in our case.
\end{proof}

Say that $\vnu \le_S \vmu$ when $\vnu^{(i)} \le \vmu^{(i)}$ for all $i \in
\Z/n$. If $D$ and $E$ are the diskoids of $w(\vnu)$ and $w(\vmu)$, then
this condition says that $d_D(p,q) \le d_E(p,q)$ for every two vertices
on their common boundary.  \thm{th:unitriang} follows by combining
\thm{th:unitriangweak} with \lem{l:cyclic}.

We define a subset $U \subseteq Z$ as follows:
$$U = \{ (L_i)_{i \in \Z/n} \in F(\vlambda) | d(L_i, L_j) = \mu^{(i)}_j \}.$$
\lem{l:cyclic} shows that $U$ is dense in $Z$.  The following proposition
then completes the proof of \thm{th:main}.

\begin{proposition} Restricting the configuration to the boundary gives
an isomorphism
$$\pi: Q_g(D) \stackrel{\cong}{\longto} U.$$
\end{proposition}

\begin{proof} By definition, $U$ consists of those configurations of $D$
that preserve all distances between boundary vertices.  By \lem{l:complete},
these are exactly the configurations that preserve all distances in $D$.
\end{proof}

If $f \in Q_g(D)$ is a global isometry, then in particular it is an embedding
of $D$ into the affine building $\Delta$.  This has an interesting area
consequence.

\begin{lemma} Let $K$ be a 2-dimensional simplicial complex with trivial
homology, $H_*(K,\Z) = H_*(\pt)$.  Then every simplicial 1-cycle $\alpha$
in $K$ is the homology boundary of a unique 2-chain $\beta$.
\label{l:homol} \end{lemma}
\begin{proof} If $\beta_1$ and $\beta_2$ are two such 2-chains, then
$\beta_1-\beta_2$ is closed and therefore null-homologous.  Since $K$
has no 3-simplices, the only way for $\beta_1$ and $\beta_2$ to be
homologous is if they are equal.
\end{proof}

\begin{theorem} If a $\CAT(0)$, type $A_2$ diskoid $D$ is embedded in an
affine building $\Delta$, then it is the unique least area diskoid that
extends the embedding of its boundary $P$.
\end{theorem}

\begin{proof} Let $f$ be the embedding.  Then $f_*([D])$ is a 2-chain whose
1-norm is the area of $D$.  If $f':D' \to \Delta$ is another extension of
$P$, then $f'_*([D']) = f_*([D])$ and the area of $D'$ cannot be smaller
than the area of $D$.  Moreover, if they have equal area, then $f^{-1}
\circ f$ is a bijection between the faces of $D'$ and the faces of $D$.
The faces of $D'$ must be connected in the same way as those of $D$, and
attached to $P$ in the same way, because each edge in $\Delta$ has at most
two faces of $f(D)$.
\end{proof}

By contrast, the $A_2$ spider relations \eqref{e:a2spider} reduce the
area of a diskoid.  The following proposition is easy to check, as well
as inevitable given \prop{l:homol} and \thm{th:main}:

\begin{proposition} If $w$ is a web with a face with 2 or 4 sides, so
that the dual diskoid $D$ has a vertex with 2 or 4 triangles, then in
any configuration $f:D \to \Gr$ these triangles land on top of each other
in pairs.
\label{p:ontop} \end{proposition}

\prop{p:ontop} thus motivates the relations \eqref{e:a2spider}
as moves that locally remove area from a configuration $f$.

\subsection{Web bases are not Satake}
\label{s:notsatake}

In \sec{s:unitriang}, we showed that the transformation between the
web basis and the Satake basis is unitriangular with respect to the
given order. Thus it is reasonable to ask if this transformation is
the identity.  As with Lustzig's dual canonical basis, there is an early
agreement between the two. For any web with no internal faces, that is,
whose dual diskoid has no internal vertices, the image of the map $\pi$
is $\bar{Q(A(\vlambda, \vmu))}$ by \thm{th:smooth}, and $\pi$ is injective.
It follows from \cor{c:support} and \lem{l:coeff1} that $[\bar{Q(A(\vlambda,
\vmu))}]$ is the web vector.

Now consider the following web $w(\vmu)$, with the indicated base point:
$$w(\vmu) = \begin{tikzpicture}[scale=0.4,baseline]
\foreach \angle in {0,120,240}
{
    \begin{scope}[rotate=\angle,web]
    \foreach \x/\y in {-3.5/-.866,-3.5/.866,-2/-1.732,-2/0,-2/1.732,
        -.5/-2.598,-.5/-.866,-.5/.866,-.5/2.598,1/-1.732,1/0,1/1.732,
        2.5/-.866,2.5/.866}
    \draw[midto] (\x,\y) -- ++(1,0);
    \end{scope}
}
\draw (0,0) circle (3.605);
\fill (240:3.605) circle (.1);
\end{tikzpicture}$$
In \cite{Kuperberg:notdual}, it was shown that this is the first web whose
invariant vector is not dual canonical.  This is the web associated with
the minuscule path
\begin{multline*}
\vmu = (0, \omega_1, \omega_1+\omega_2, \omega_1+2\omega_2, 3\omega_2,
    \omega_1+3\omega_2, \\
    2\omega_1+2\omega_2, 3\omega_1+\omega_2, 3\omega_1,
    2\omega_1+\omega_2, \omega_1+\omega_2, \omega_2, 0)
\end{multline*}
of type
$$ \vlambda = (\omega_1, \omega_2, \omega_2, \omega_1, \omega_1, \omega_2,
\omega_2, \omega_1, \omega_1, \omega_2, \omega_2, \omega_1).
$$

Let
$$\vnu=(0,\omega_1,0,\omega_2,0,\omega_1,0,\omega_2,0,\omega_1,0,\omega_2,0).$$
This is another minuscule path also of type $\vlambda$; the
corresponding web $w(\vnu)$ is much simpler and is both a Satake vector
and a dual canonical vector:
$$w(\vnu) = \begin{tikzpicture}[scale=0.4,baseline]
\foreach \angle in {0,120,240}
{
    \begin{scope}[rotate=\angle,web]
    \draw[midto] (1,3.464) .. controls (.5,2.598)
        and (-.5,2.598) .. (-1,3.464);
    \draw[midto] (1,-3.464) .. controls (.5,-2.598)
        and (-.5,-2.598) .. (-1,-3.464);
    \end{scope}
}
\draw (0,0) circle (3.605);
\fill (240:3.605) circle (.1);
\end{tikzpicture}$$

In \cite{Kuperberg:notdual}, it was shown that
$$\Psi(w(\vmu)) = b(\vmu) + b(\vnu),$$
where $b(\vmu)$ denotes the dual canonical basis vector
indexed by $\vmu$.

\begin{theorem} Let $w(\vmu)$, $\vlambda$, $\vmu$, and $\vnu$ be as
above.  Then the invariant vector $\Psi(w(\vmu))$ is not in the Satake
basis.  More precisely, it has a coefficient of 2 for the basis vector
$[\bar{Q(A(\vlambda, \vnu))}]$.
\end{theorem}

\begin{proof}
We will show that the general fibre of $\pi$ over $Q(A(\vlambda, \vnu))$
is of size 2. We give the faces of the web the following labels:
$$\begin{tikzpicture}[scale=0.6]
\foreach \angle in {0,120,240}
{
    \begin{scope}[rotate=\angle,web]
    \foreach \x/\y in {-3.5/-.866,-3.5/.866,-2/-1.732,-2/0,-2/1.732,
        -.5/-2.598,-.5/-.866,-.5/.866,-.5/2.598,1/-1.732,1/0,1/1.732,
        2.5/-.866,2.5/.866}
    \draw[midto] (\x,\y) -- ++(1,0);
    \end{scope}
}
\node at ($(240:2)+(180:2)$) {$p_1$};
\node at ($(180:2)+(120:2)$) {$\ell_1$};
\node at ($(120:2)+(60:2)$) {$p_2$};
\node at ($(60:2)+(0:2)$) {$\ell_2$};
\node at ($(0:2)+(300:2)$) {$p_3$};
\node at ($(300:2)+(240:2)$) {$\ell_3$};
\node at ($(240:1)+(180:1)$) {$\ell_1'$};
\node at ($(180:1)+(120:1)$) {$p_1'$};
\node at ($(120:1)+(60:1)$) {$\ell_2'$};
\node at ($(60:1)+(0:1)$) {$p_2'$};
\node at ($(0:1)+(300:1)$) {$\ell_3'$};
\node at ($(300:1)+(240:1)$) {$p_3'$};
\node at (0,0) {$c$};
\end{tikzpicture}$$
If $f\in Q(D(\vlambda, \vmu))$ then $\pi(f)\in Q(A(\vlambda,\vnu))$ if
and only if $f$ assigns $p_i\in \Gr(\omega_1)$ and $\ell_i\in \Gr(\omega_2)$
on those faces and assigns $t^0 \in \Gr(0)$ to all empty faces. In order
to determine the fibre of $\pi$ over a point in $Q(A(\vlambda,\vnu))$
we must calculate the possible choices for $p_i'$, $\ell_i'$ and $c$
satisfying the appropriate conditions. Since $p_i\in \Gr(\omega_1)$ and
$\ell_i\in \Gr(\omega_2)$, this forces $p_i'\in \Gr(\omega_1)$ and $\ell_i'\in
Gr(\omega_2)$ and $c\in\bar{\Gr(\omega_1+\omega_2)}$. We can think of the
points of $\Gr(\omega_1)$ and $\Gr(\omega_2)$ as, respectively, the points
and lines in $\mathbb{CP}^2$. Then the conditions given by the edges of
the web are as following: $p_i'$ is a point on the line $\ell_i$ and $\ell_i'$
is a line containing the points $p_i$, $p_{i-1}'$ and $p_i'$.

\begin{fullfigure*}{f:twosolutions}
    {The two solutions to the problem for the given $\ell_i$ and $p_i$}
\subfloat[]{\begin{tikzpicture}[scale=0.55]
\useasboundingbox (-6,-5) rectangle (6,5);
\clip (-5,-5) rectangle (5,5);
\draw[domain=-5:5] plot(\x,{(-1-1.81*\x)/0.99});
\draw[domain=-5:5] plot(\x,{(--3.51--0.22*\x)/2.34});
\draw[domain=-5:5] plot(\x,{(-0.06-2.03*\x)/-1.35});
\draw (-2.56,4.5) node {$\ell_1$};
\draw (-4.5,1.55) node {$\ell_2$};
\draw (2.50,4.5) node {$\ell_3$};

\draw[blue,domain=-5:5] plot(\x,{(--2.34--0.37*\x)/3.3});
\draw[blue,domain=-5:5] plot(\x,{(--2.23-3.24*\x)/0.89});
\draw[blue,domain=-5:5] plot(\x,{(-1.51-0.91*\x)/-1.16});
\draw (4.5,0.75) node {$\ell'_1$};
\draw (-4.5,-1.65) node {$\ell'_2$};
\draw (-0.2,4.5) node {$\ell'_3$};

\fill[blue] (-0.89,0.61) circle (.1); \draw (-0.98,0.1) node {$p'_1$};
\fill[blue] (0.27,1.52) circle (.1); \draw (0.51,2.2) node {$p'_2$};
\fill[blue] (0.48,0.76) circle (.1); \draw (1.1,0.45) node {$p'_3$};

\fill[red] (-1.31,1.38) circle (.1); \draw (-1.00,1.7) node {$e_1$};
\fill[red] (1.04,1.59) circle (.1); \draw (1.75,2) node {$e_2$};
\fill[red] (-0.32,-0.43) circle (.1); \draw (0.2,-0.5) node {$e_3$};

\fill[darkgreen] (-4.19,0.24) circle (.1); \draw (-4.1,-0.24) node {$p_1$};
\fill[darkgreen] (-1.89,-0.16) circle (.1); \draw (-1.87,-0.7) node {$p_2$};
\fill[darkgreen] (1.37,-2.48) circle (.1); \draw (1.9,-2.5) node {$p_3$};
\end{tikzpicture}}
\subfloat[]{\begin{tikzpicture}[scale=0.55]
\useasboundingbox (-6,-5) rectangle (6,5);
\clip (-5,-5) rectangle (5,5);
\draw[domain=-5:5] plot(\x,{(-1-1.81*\x)/0.99});
\draw[domain=-5:5] plot(\x,{(--3.51--0.22*\x)/2.34});
\draw[domain=-5:5] plot(\x,{(-0.06-2.03*\x)/-1.35});
\draw (-2.56,4.5) node {$\ell_1$};
\draw (4.5,2.4) node {$\ell_2$};
\draw (2.50,4.5) node {$\ell_3$};

\draw[blue,domain=-5:5] plot(\x,{(-6.07-1.7*\x)/4.43});
\draw[blue,domain=-5:5] plot(\x,{(-3.34-1.4*\x)/2.12});
\draw[blue,domain=-5:5] plot(\x,{(-5.7-2.58*\x)/4.33});
\draw (4.5,-2.67) node {$\ell'_1$};
\draw (4.5,-3.6) node {$\ell'_2$};
\draw (3.7,-4.5) node {$\ell'_3$};

\fill[blue] (0.24,-1.46) circle (3pt); \draw (0.58,-1) node {$p'_1$};
\fill[blue] (-4.09,1.12) circle (3pt); \draw (-3.85,1.6) node {$p'_2$};
\fill[blue] (-0.75,-1.08) circle (3pt); \draw (-0.7,-1.7) node {$p'_3$};

\fill[red] (-1.31,1.38) circle (3pt); \draw (-1.00,1.7) node {$e_1$};
\fill[red] (1.04,1.59) circle (3pt); \draw (1.75,2) node {$e_2$};
\fill[red] (-0.32,-0.43) circle (3pt); \draw (0.2,-0.5) node {$e_3$};

\fill[darkgreen] (-4.19,0.24) circle (3pt); \draw (-4.1,-0.24) node {$p_1$};
\fill[darkgreen] (-1.89,-0.16) circle (3pt); \draw (-1.87,0.3) node {$p_2$};
\fill[darkgreen] (1.37,-2.48) circle (3pt); \draw (1.55,-3) node {$p_3$};
\end{tikzpicture}}
\end{fullfigure*}

Suppose that either the $p_i$ are not collinear and or the $\ell_i$ are not
concurrent.  Then by the duality of points and lines, we may assume that
the $\ell_i$ are not concurrent. Let $e_i$ be the intersection of $\ell_i$ and
$\ell_{i+1}$.  Then we can express the points $p_i'$ in barycentric coordinates
given by $e_i$:
\begin{align*}
p_1'=(t_1,0,1-t_1) \\ p_2'=(1-t_2,t_2,0) \\ p_3'=(0,1-t_3,t_3).
\end{align*}
Note that by doing this we restrict ourselves to an affine subspace of
$\P^2$, so we may lose, but we don't gain solutions. The collinearity
condition results in the equations
$$p_i=(1-s_i)p_i'+s_ip_{i-1}.$$
Solving this problem amounts to solving
\begin{align*}
(1-s_1)t_1 &= p_{11} & s_1(1-t_3) &= p_{12} \\
(1-s_2)t_2 &= p_{22} & s_2(1-t_1) &= p_{23} \\
(1-s_3)t_3 &= p_{33} & s_3(1-t_2) &= p_{31},
\end{align*}
where $p_{ij}$ are the barycentric coordinates of the $p_i$. If none of
these coordinates are $0$, then we can eliminate all but one variable to get
the relation
$$t_1=\frac{p_{11}}{1-\frac{p_{12}}{1-\frac{p_{33}}
{1-\frac{p_{31}}{1-\frac{p_{22}}{1-\frac{p_{23}}{1-t_1}}}}}}.$$
The right side of this equation is a composition of fractional
linear transformations that condenses to a single fractional linear
transformation
$$t_1 = \frac{\alpha_{11} t_1 + \alpha_{12}}{\alpha_{21}t_1 + \alpha_{22}}$$
with generic coefficients.  Thus, generically,
we obtain a quadratic equation for $t_1$ with 2 solutions.

It remains to determine the face $c$, which lies in
$\Gr(\omega_1+\omega_2)$. If $c\not\in \Gr(0)$, then the conditions given
by the edges of the web would be $p_i'=p_j'$ and $\ell_i'=\ell_j'$ for all
$i,j$ which cannot happen since either $p_i$ are not collinear or $\ell_i$
are not concurrent. Thus for any solution of the above equations, we get
exactly one element in $Q(D(\vlambda, \vmu))$. And for any generic point
$p \in Q(A(\vlambda, \vnu))$, the fibre $\pi^{-1}(p)$ has 2 points.

Let $X$ denote the closure in $Q(D(\vlambda, \vmu))$ of the union of all
fibres $\pi^{-1}(p)$ with 2 points.  Then $X$ is either a component of
$Q(D(\vlambda, \vmu))$ or a union of two components.  Moreover, $X$ contains
all components of $Q(D(\vlambda, \vmu))$ which map onto $Q(A(\vlambda,
\vnu))$.  Since the above argument shows that the scheme-theoretic fibre of
$\pi$ over a general point of $Q(A(\vlambda, \vnu))$ is two reduced points,
we also know that $X$ is generically reduced.  Hence the coefficient of
$[X]$ in the homology class $c(w)$ from \thm{th:homclass2} is 1.  Since the
map $\pi: X \to Q(A(\vlambda, \vnu))$ is of degree 2 and since $X$
contains all components mapping to $Q(A(\vlambda, \vnu))$, the coefficient of
$[Q(A(\vlambda, \vnu))]$ in $\pi_*(c(w))$ is 2.  In particular, $\pi_*(c(w))$
differs from $[\bar{Q(A(\vlambda, \vmu))}]$, as desired.
\end{proof}

In fact, we suspect that $Q(D(\vlambda, \vmu))$ only has two components,
which would imply that
$$\Psi(w(\vmu)) = [\bar{Q(A(\vlambda, \vmu))}]
    + 2 [\bar{Q(A(\vlambda,\vnu))}].$$
Otherwise, $\Psi(w(\vmu))$ has these two terms and perhaps others.
Either way, the coefficient of 2 is different from what arises in the dual
canonical basis \cite{Kuperberg:notdual}:
$$\Psi(w(\vmu)) = b(\vmu) + b(\vnu).$$
Thus,

\begin{theorem} The geometric Satake bases for invariants of $G = \SL(3)$
are eventually not dual canonical.
\label{th:notdual} \end{theorem}

This is not such a surprising statement in light of the well-known fact that
the canonical and semicanonical basis do not coincide (as a consequence of
the work of Kashiwara-Saito \cite{KS:crystal}).  In both \thm{th:notdual}
and in the canonical/semicanonical situation, a homology basis does not
coincide with a basis defined using a bar-involution.  The analogy between
these two results could perhaps be made precise using skew Howe duality
($\SL(3)$, $\SL(n)$-duality).

It is known that $\Psi(w(\vmu))$ is the first basis web that is not
dual canonical, \ie, the only basis web up to rotation with 12 or fewer
minuscule tensor factors.  We conjecture that it is also the first basis
web for $\SL(3)$ that is not geometric Satake.  Equivalently, we conjecture
that all three bases first diverge at the same position.
\begin{question} For arbitrary $G$,
is the dual canonical basis of an invariant space $\Inv_G(V(\lambda))$
positive unitriangular in the geometric Satake basis?
\end{question}

\section{Euler convolution of constructible functions}
\label{s:euler}

In this section, we switch from convolution in homology to convolution
in constructible functions.  The idea of defining convolution algebras
using constructible functions is common in geometric representation theory
(see for example \cite{Lusztig:constructible}).

More specifically, we will define a new category $\econv(\Gr)_0$ which
conjecturally is equivalent to $\rep^u_{-1}(G)_\min$, and we will prove
this conjecture for $G = \SL(2)$ and $G = \SL(3)$.  When computing invariant
vectors from webs, the construction is a state model as in \sec{s:hstate},
where the counting is done using Euler characteristic.

\subsection{Generalities on constructible functions}

If $X$ is a proper complex algebraic variety over $\C$ and $f:X \to \C$
is a constructible function, then we define the \emph{Euler characteristic
integral} (see \cite{MacPherson:chern} or \cite{Joyce:constructible})
$$\int_X f d\chi \in \C$$
by linear extension starting with the characteristic functions of closed
subvarieties.  Namely, if $f = f_Y$ is the characteristic function of a
closed subvariety $Y \subseteq X$, then we define
$$\int_X f_Y d\chi \defeq \chi(Y).$$

If $\pi:X \to Y$ is a proper morphism between algebraic varieties and $f:X
\to \C$ is a constructible function on $X$, then we define the push-forward
of $f$ under $\pi$ by integration along fibres:
$$(\pi_*f)(p) \defeq \int_{\pi^{-1}(p)} f d\chi.$$
If $\cf(X)$ denotes the vector space of constructible functions on $X$,
this pushforward is then a linear map:
$$\pi_*:\cf(X) \to \cf(Y).$$

The following result is well-known --- see for example
\cite[Prop. 1]{MacPherson:chern} and \cite[Thm. 3.8]{Joyce:constructible}.

\begin{theorem} The Euler characteristic integral push-forward of
constructible functions is a well-defined covariant functor from the
category of proper morphisms between algebraic varieties over $\C$, to
the category of complex vector spaces.
\label{th:efunctor} \end{theorem}

\subsection{Construction of the categories}

Given $G$ simple and simply-connected as before, we can define a pivotal
category $\econv(\Gr)$ in a similar fashion to $\hconv(\Gr)$, except that
we will replace homology with constructible functions throughout.

The objects of $\econv(\Gr)$ are the $\Gr(\vlambda)$, where $\vlambda$ is
a sequence of minuscule weights.  As in $\hconv(\Gr)$, the tensor product
is defined by convolution.

We define the invariant space of $\Gr(\vlambda)$ to be the vector space of
constructible functions on the Satake fibre:
$$\Inv_{\econv(\Gr)}(\Gr(\vlambda)) \defeq \cf(F(\vlambda)).$$
The hom spaces are defined in an equivalent way:
$$\Hom_{\econv(\Gr)}(\Gr(\vlambda),\Gr(\vmu))
    \defeq \cf(Z(\vlambda,\vmu)).$$
We define the convolution of two hom spaces by convolution as in
$\hconv(\Gr)$.  We could proceed exactly as in $\hconv(\Gr)$, but the
``local'' nature of constructible functions allows us a simpler definition.

Fix three minuscule sequences $\vlambda, \vmu, \vnu$.  Let $\Gamma$ be a graph
homeomorphic to a theta ($\theta$) with three arcs that are polylines of
type $\vlambda$,$\vmu$, and $\vnu$ with a common base point:
$$\Gamma \;\;=\;\; \begin{tikzpicture}[baseline]
\draw[darkred,semithick] (0,0) -- (1,1) -- (2,.5) -- (3,1) -- (3.5,0)
    -- (3,-1) -- (1,-1) -- cycle;
\draw[darkred,semithick] (0,0) -- (1,0) -- (2.5,-.5) -- (3.5,0);
\draw[darkred,semithick,midto] (1,0) -- (2.5,-.5);
\draw[darkred,semithick,midto] (1,1) -- (2,.5);
\draw[darkred,semithick,midto] (1,-1) -- (3,-1);
\fill[darkred] (0,0) circle (.07);
\path[draw=darkred,fill=white] (3.5,0) circle (.07);
\draw[anchor=north] (1.75,-1.1) node {$\vlambda$};
\draw[anchor=north] (1,-.1) node {$\vmu$};
\draw[anchor=south] (1.75,.9) node {$\vnu$};
\end{tikzpicture}$$
Then there are projections
\begin{align*}
\pi_{\vlambda,\vmu}:Q(\Gamma) &\to Z(\vlambda,\vmu) \\
\pi_{\vlambda,\vnu}:Q(\Gamma) &\to Z(\vlambda,\vnu) \\
\pi_{\vmu,\vnu}:Q(\Gamma) &\to Z(\vmu,\vnu)
\end{align*}
Given
\begin{align*}
f &\in \Hom_{\econv(\Gr)}(\Gr(\vlambda),\Gr(\vmu)) \\
g &\in \Hom_{\econv(\Gr)}(\Gr(\vmu),\Gr(\vnu)),
\end{align*}
we can define their composition by Euler characteristic integration over
configurations of the middle polyline $L(\vmu)$, and using the fact that
constructible functions pull back and multiply as well as push forward:
$$g \circ f \defeq (\pi_{\vlambda,\vnu})_*(\pi_{\vlambda,\vnu}^*(f)
\pi_{\vmu,\vnu}^*(g)).$$
It is routine to check that these structures define a pivotal category.

The hom spaces in the category $\econv(\Gr)$ are too large for our purposes.
We will restrict them by just looking at those constructible functions
generated by the constant functions on the Satake fibres corresponding to
trivalent vertices.  More precisely, define a pivotal functor
$$E:\fsp(G) \to \econv(\Gr)$$
which takes the generating vertex in
$$\Inv_{\fsp(G)}(\lambda,\mu,\nu),$$
to the identity function on $F(\lambda,\mu,\nu)$.  Again, $\lambda$, $\mu$,
and $\nu$ are all minuscule and we are assuming that there is a vertex, so
$$\Inv_G(V(\lambda, \mu, \nu)) \ne 0.$$
Let $\econv(\Gr)_0$ denote the image of the functor $E$; it has
the same objects as $\econv(\Gr)$, but smaller hom spaces.

\subsection{Equivalence with the representation category}

Before stating the main conjecture and result, we can describe more
explicitly how the functor $E$ expresses an Euler characteristic state model.
The following result can be seen by chasing through the definitions.

\begin{proposition} Given a web $w \in \fsp(G)$ with boundary $\vlambda$
and dual diskoid $D$, $E(w)$ is the function on the Satake fibre
$F(\vlambda)$ whose value at $p \in F(\vlambda)$ is $\chi(\pi^{-1}(p))$.
(Here $\pi:Q(D) \to F(\vlambda)$ is the map which restricts a diskoid
configuration to its boundary.)
\end{proposition}

So we are indeed producing a function which counts (using Euler
characteristic) ways to extend the boundary configuration to a diskoid
configuration.

We are now ready to formulate our alternate version of the geometric
Satake correspondence.

\begin{conjecture} There is an equivalence of pivotal categories:
$$\econv(\Gr(\Gv))_0 \cong \rep^u_{-1}(G)_\min.$$
\label{c:eulerrep} \end{conjecture}

Recall from Sections~\ref{s:signs} and \ref{s:examples} that
$\rep^u_{-1}(G)_\min$ and $\rep^u(G)_\min$ have the same information except
for a sign correction.  We offer the following corollary of \conj{c:eulerrep}
as a stand-along conjecture.

\begin{conjecture}  Let $w$ be any closed web with dual diskoid $D$.  Then
$$\Psi(w) = \pm \chi(Q(D)).$$
\label{c:eulerclosed} \end{conjecture}

Here $\Psi(w)$ denotes the value of $w$ in the pivotal category
$\rep^u(G)_\min$ and the sign comes as a result of the sign correction
between $\rep^u_{-1}(G)_\min$ and $\rep^u(G)_\min$.

\begin{theorem} \conj{c:eulerrep} holds when $G = \SL(2)$ and $G = \SL(3)$.
\end{theorem}

\begin{proof} We will first argue the more difficult case $G = \SL(3)$.
We argue by checking the skein relations of $\spd_{-1}(\SL(3))$.  The first
two skein relations,
\begin{align*}
\begin{tikzpicture}[baseline=-.5ex,web,allow upside down]
\draw[midto,rotate=180] (0,0) circle (.4);
\end{tikzpicture}\;\; &= \;\; 3 \\
\begin{tikzpicture}[baseline=-.5ex,web]
\draw[midto] (-1.25,0) -- (-.5,0);
\draw[midto] (.5,0) -- (1.25,0);
\draw (.5,0) arc (30:150:.577); \draw[midto] (0,.288) -- (-.01,.288);
\draw (.5,0) arc (330:210:.577); \draw[midto] (0,-.288) -- (-.01,-.288);
\end{tikzpicture}\;\; &= \;\; 2 \;\;
\begin{tikzpicture}[baseline=-.5ex,web]
\draw[midto] (-.5,0) -- (.5,0);
\end{tikzpicture}
\end{align*}
are straightforward, because the relevant fibres are always $\P^2$ and
$\P^1$, respectively.  The third skein relation,
$$\begin{tikzpicture}[baseline=-.5ex,web]
\draw (.35,.35) arc (60:120:.7) arc (150:210:.7)
    arc (240:300:.7) arc (-30:30:.7);
\draw[midfrom] (0,.444) -- (.01,.444);
\draw[midfrom] (-.444,0) -- (-.444,-.01);
\draw[midfrom] (0,-.444) -- (-.01,-.444);
\draw[midfrom] (.444,0) -- (.444,.01);
\draw[midto] (.35,.35) -- (.707,.707);
\draw[midfrom] (-.35,.35) -- (-.707,.707);
\draw[midto] (-.35,-.35) -- (-.707,-.707);
\draw[midfrom] (.35,-.35) -- (.707,-.707);
\end{tikzpicture}\;\; =\;\;
\begin{tikzpicture}[baseline=-.5ex,web]
\draw (.5,.5) arc (315:225:.707);
\draw[midto] (0,.293) -- (.01,.293);
\draw (-.5,-.5) arc (135:45:.707);
\draw[midto] (0,-.293) -- (-.01,-.293);
\end{tikzpicture}\;\; + \;\
\begin{tikzpicture}[baseline=-.5ex,web]
\draw (.5,-.5) arc (225:135:.707);
\draw[midto] (.293,0) -- (.293,.01);
\draw (-.5,.5) arc (45:-45:.707);
\draw[midto] (-.293,0) -- (-.293,-.01);
\end{tikzpicture},$$
is a little bit more work. The diskoid dual to the left side consists
of four triangles.  The configuration space of the quadrilateral
$P(\omega_1,\omega_2,\omega_1,\omega_2)$ has two components, corresponding
to the two ways to collapse the quadrilateral to two edges.  In each case,
there is a unique extension to the diskoid in which the diskoid collapses to
two triangles.  The remaining case that should be checked is the intersection
of the two components in which the quadrilateral collapses to a single edge.
In this case the fibre is $\P^1$, because there is a $\P^1$ of ways to
extend the edge to a triangle, and the diskoid can collapse onto this
triangle.  Thus the local Euler characteristic at the intersection is 2,
which matches the sum on the right side of the skein relation.

Thus, the image of $E$ is either equivalent to $\rep^u_{-1}(\SL(3))_\min$,
or it is a quotient.  However, $\rep^u_{-1}(\SL(3))_\min$ is simple
as a linear-additive, pivotal category, because the pairing of dual
invariant spaces is non-degenerate.  (Or, \thm{th:unitriang} also
implies that basis webs are linearly independent after applying $E$
because of unitriangularity.)  Therefore the image of $E$ is equivalent
to $\rep^u_{-1}(\SL(3))_\min$ itself.

In the case $G = \SL(2)$, we only need to check the skein relation
\eqref{e:tl} with $q=-1$:
$$\begin{tikzpicture}[baseline=-.5ex,web]
\draw (0,0) circle (.4);
\end{tikzpicture} \;\;= \;\;2.$$
In this case the diskoid of the left side is a based edge, the diskoid of
the right side is a point, the fibre is $\P^1$, and its Euler characteristic
is 2 as desired.
\end{proof}

It should also be possible to prove \conj{c:eulerrep} when $G = \SL(m)$.
The idea is to use the geometric skew Howe duality of Mirkovi\'c and
Vybornov \cite{MV:bdg} and the ideas in \cite[Sec. 6]{Kamnitzer:lectures}
to express this conjecture in terms of constructible functions on quiver
varieties for the Howe dual $\SL(n)$.  Then we are in a position to apply
Nakajima's work from \cite[Sec. 10]{Nakajima:kacmoody}.  Note that this
approach does not make use of the geometric Satake correspondence.

\subsection{Relationship with homological convolutions}

A constructible function is constant on a dense open subset of any
irreducible variety.  If $X$ is an irreducible variety, we write
$f(X)$ for the value of $f$ on this dense open subset.

We can define a non-functor $\Xi$ from $\econv(\Gr)$ to $\hconv(\Gr)$
as follows.  On objects, $\Xi$ is the identity, while on morphisms we define
$$\Xi:\cf(Z(\vlambda, \vmu)) \longto H_\top(Z(\vlambda, \vmu))$$
by the formula
$$\Xi:f \mapsto \sum_{X \in \Irr(Z(\vlambda, \vmu))} f(X)[X].$$
The map $\Xi$ is not a functor because it does not respect convolution
(as some simple examples show). However, we offer the following tentative
conjecture.

\begin{conjecture} The map $\Xi$ between hom spaces restricts to an
equivalence of pivotal categories from $\econv(\Gr)_0$ to $\hconv(\Gr)$
up to a sign correction of the tensor and pivotal structures.
\label{c:eulerhom} \end{conjecture}

This conjecture implies \conj{c:eulerrep} one because this conjectured
equivalence is compatible with the functors from $\fsp(G)$.  This conjecture
would also imply the following simple formula for the expansion of the
invariant vectors coming from webs in the Satake basis which generalizes
\conj{c:eulerclosed}.

\begin{conjecture}[Corollary of \conj{c:eulerhom}]
Let $w$ be a minuscule web with boundary $\vlambda$ and dual diskoid
$D$.  Then we can expand $\Psi(w)$ in the Satake basis as
$$\Psi(w) = \pm \sum_{X \in \Irr(F(\vlambda))} \chi(\pi^{-1}(x))[X],$$
where $x$ is a generic point of each $X$, and $\pi:Q(D) \to F(\vlambda)$
is the restriction map from a diskoid configuration to its boundary.
\end{conjecture}

As partial evidence for \conj{c:eulerhom}, we note that a similar result
has been conjectured in the quiver variety setting.

\section{Future work}

This article is hopefully only the beginning of an investigation into
configuration spaces of diskoids and their relations to presented pivotal
categories, or spiders.

\subsection{Basis webs for $\SL(n)$}

In future work, the first author will establish the following
generalization of \thm{th:sl3basis} and \thm{th:unitriang} to $\SL(n)$.

\begin{theorem} Given a sequence of minuscule weights $\vlambda$ of $\SL(n)$,
there is a map $w(\vmu)$ from minuscule path $\vmu$ of type $\vlambda$
to webs. The image of these webs in $\Inv(V(\vlambda))$ are a basis,
and the change of basis to the Satake basis is upper unitriangular with
respect to the partial order on minuscule paths.
\label{th:westbury} \end{theorem}

The geometric results of the current article are used to establish that
the webs $w(\vmu)$ form a basis and as far as we are aware no elementary
proof is available.  This is in sharp contrast to the $\SL(3)$ case
where the basis webs were originally established by elementary means.
The webs $w(\vmu)$ themselves are constructed combinatorially using the
idea of Westbury triangles \cite{Westbury:so7} and \cite{Westbury:general}.
Recently, Westbury has combinatorially obtained \thm{th:westbury} for
the case of a tensor product of standard representation and their duals.

Kim \cite{Kim:thesis} (for $n=4$) and Morrison \cite{Morrison:diagram}
(for general $n$) conjecture a set of generating relations for kernel of
$\fsp(\SL(n)) \to \rep(\SL(n))_\min$.  Using \thm{th:westbury}, we hope
to establish Kim's and Morrison's conjectures.

\subsection{Other rank 2 groups}

Since there are established definitions of spiders for $B_2$ and $G_2$,
it seems quite possible that the results in this paper could be generalized
to these two cases, but there are two important problems to resolve.  First,
the vertex set of the corresponding affine buildings are no longer simply the
points of the affine Grassmannian. Second, since we want to study $\rep(G)$
rather than $\rep(G)_\min$, it is necessary to look at webs labelled not just
by minuscule weights but by fundamental weights. When $G$ is not $\SL(n)$,
it is no longer the case that all fundamental weights are minuscule;
thus the results of this paper would need to be extended to cover this case.

\subsection{Other discrete valuation rings}

Our results in this article apply only to the affine Grassmannians of the
discrete valuation ring $\cO = \C[[t]]$.  In fact, the affine Grassmannian
$\Gr$ exists (as a set) and Bruhat-Tits building $\Delta$ exists and is
$\CAT(0)$ for any complete discrete valuation ring $\cO$.  It is a well-known
open problem to state and prove a geometric Satake correspondence in this
setting, it is only known in the equal characteristic case $\cO = k[[t]]$
for a field $k$.  Since the building geometry is so similar for all choices
of $\cO$, our results could be interpreted as (further) evidence that a
geometric Satake correspondence exists for all $\cO$.

\subsection{Webs in surfaces}

Another possible generalization is from webs in disks to webs in surfaces.
If $\Sigma$ is a closed surface and $G$ has rank 1 or 2, there is an
analogous basis of non-elliptic webs on $\Sigma$ \cite{SK:confluence}, which
are equivalent to $\CAT(0)$ triangulations.  (Or, $\Sigma$ can have boundary
circles with marked points, but the closed case is especially interesting.)
This web basis is a basis of the skein module of $\Sigma \times [0,1]$, which
is also the coordinate ring of the variety of representations $\pi_1(\Sigma)
\to G$.  Our results suggest an interpretation of this coordinate ring in
terms of certain simplicial maps from the universal cover of $\Sigma$ to
the affine building $\Delta$.  This should be related to the conjectures
of Fock-Goncharov \cite{FG:Teichmuller}.

\subsection{Categorification}

We would also like to apply our results to categorification and knot
homology.  According to the philosophy of \cite{CK:coherent2}, to each web
$w$ with dual diskoid $D$ and boundary $\vlambda$, we should associate an
object $A(w)$ in the derived category of coherent sheaves on $\Gr(\vlambda)$.
When the configuration space $Q(D)$ has the expected dimension, then $A(w)$
should be the pushforward $\pi_*(\cO_{Q(D)})$ of the structure sheaf
of $Q(D)$.  It would also be nice to understand foams (as introduced by
Khovanov \cite{Khovanov:sl3}) in this language.   In particular, it would
be interesting to consider the configuration spaces of duals of foams.
Some ideas in this direction have been pursued by Frohman.

\subsection{Quantum groups}

Finally, developing a $q$-analogue of our theory is also an open problem.
As mentioned earlier there is a functor from the free spider $\fsp(G)$ to
$\rep_q(G)$, the representation category of the quantum group for any $q$
not a root of unity.  However, our geometric Satake machinery only applies
in the case when $q = 1$.  Hopefully, we can extend to general $q$ using
the quantum geometric Satake developed by Gaitsgory \cite{Gaitsgory:twisted}.


\begin{thebibliography}{10}

\bibitem{BD:hitchin}
Alexander Beilinson and Vladimir Drinfeld, \emph{Quantization of {Hitchin's}
  integral system and {Hecke} eigensheaves}, preprint.

\bibitem{BT:local1}
Fran\c{c}ois Bruhat and Jacques Tits, \emph{Groupes r\'eductifs sur un corps
  local}, Inst. Hautes \'Etudes Sci. Publ. Math. \textbf{41} (1972), 5--251.

\bibitem{CK:coherent2}
Sabin Cautis and Joel Kamnitzer, \emph{Knot homology via derived categories of
  coherent sheaves. {II}. {$\mathfrak{sl}_m$} case}, Invent. Math. \textbf{174}
  (2008), no.~1, 165--232, \mbox{arXiv:0710.3216}.

\bibitem{CG:complex}
Neil Chriss and Victor Ginzburg, \emph{Representation theory and complex
  geometry}, Birkh\"auser Boston Inc., 1997.

\bibitem{FG:Teichmuller}
Vladimir Fock and Alexander Goncharov, \emph{Moduli spaces of local systems and
  higher {Teichm\"uller} theory}, Publ. Math. Inst. Hautes \'Etudes Sci.
  (2006), no.~103, 1--211, \mbox{arXiv:math/0311149}.

\bibitem{FK:canonical}
Igor Frenkel and Mikhail Khovanov, \emph{Canonical bases in tensor products and
  graphical calculus for {$U_q(\mathfrak{sl}_2)$}}, Duke Math. J. \textbf{87}
  (1995), no.~3, 409--480.

\bibitem{FY:braided}
Peter~J. Freyd and David~N. Yetter, \emph{Braided compact closed categories
  with applications to low-dimensional topology}, Adv. Math. \textbf{77}
  (1989), no.~2, 156--182.

\bibitem{Fulton:intersect}
William Fulton, \emph{Intersection theory}, 2nd ed., Springer-Verlag, 1998.

\bibitem{Gaitsgory:twisted}
Dennis Gaitsgory, \emph{Twisted {Whittaker} model and factorizable sheaves},
  Selecta Math. (N.S.) \textbf{13} (2008), no.~4, 617--659,
  \mbox{arXiv:0705.4571}.

\bibitem{GL:cycles}
St\'ephane Gaussent and Peter Littelmann, \emph{{LS} galleries, the path model,
  and {MV} cycles}, Duke Math. J. \textbf{127} (2005), no.~1, 35--88,
  \mbox{arXiv:math/0307122}.

\bibitem{Ginzburg:loop}
Victor Ginzburg, \emph{Perverse sheaves on a loop group and {Langlands}
  duality}, \mbox{arXiv:alg-geom/9511007}.

\bibitem{Gromov:hyperbolic}
Mikhail Gromov, \emph{Hyperbolic groups}, Essays in group theory, Math. Sci.
  Res. Inst. Publ., vol.~8, Springer, 1987, pp.~75--263.

\bibitem{Haines:equidim}
Thomas~J. Haines, \emph{Equidimensionality of convolution morphisms and
  applications to saturation problems}, Adv. Math. \textbf{207} (2006), no.~1,
  297--327, \mbox{arXiv:math/0501504}.

\bibitem{Humphreys:gtm}
James~E. Humphreys, \emph{Introduction to {Lie} algebras and representation
  theory}, Graduate Texts in Mathematics, vol.~9, Springer-Verlag, New
  York-Heidelberg-Berlin, 1972.

\bibitem{Ivanov:personal}
Sergei Ivanov, \emph{personal communication}, 2011.

\bibitem{JS:braided}
Andr{\'e} Joyal and Ross Street, \emph{Braided tensor categories}, Adv. Math.
  \textbf{102} (1993), no.~1, 20--78.

\bibitem{Joyce:constructible}
Dominic Joyce, \emph{Constructible functions on {Artin} stacks}, J. London
  Math. Soc. (2) \textbf{74} (2006), no.~3, 583--606,
  \mbox{arXiv:math/0403305}.

\bibitem{Kamnitzer:lectures}
Joel Kamnitzer, \emph{Geometric constructions of the irreducible
  representations of $\mathrm{GL}_n$}, Geometric Representation Theory and
  Extended Affine Lie Algebras, Amer. Math. Soc., 2011, pp.~1--18,
  \mbox{arXiv:0912.0569}.

\bibitem{KLM:generalized}
Michael Kapovich, Bernhard Leeb, and John~J. Millson, \emph{The generalized
  triangle inequalities in symmetric spaces and buildings with applications to
  algebra}, 2008, \mbox{arXiv:math/0210256}, pp.~viii+83.

\bibitem{KS:crystal}
Masaki Kashiwara and Yoshihisa Saito, \emph{Geometric construction of crystal
  bases}, Duke Math. J. \textbf{89} (1997), no.~1, 9--36,
  \mbox{arXiv:q-alg/9606009}.

\bibitem{Kauffman:spinknot}
Louis~H. Kauffman, \emph{Spin networks and knot polynomials}, Internat. J.
  Modern Phys. A \textbf{5} (1990), no.~1, 93--115.

\bibitem{Khovanov:springer}
Mikhail Khovanov, \emph{Crossingless matchings and the cohomology of {$(n,n)$}
  {Springer} varieties}, Commun. Contemp. Math. \textbf{6} (2004), no.~4,
  561--577, \mbox{arXiv:math/0202110}.

\bibitem{Khovanov:sl3}
\bysame, \emph{$\mathfrak{sl}(3)$ link homology}, Algebr. Geom. Topol.
  \textbf{4} (2004), 1045--1081, \mbox{arXiv:math/0304375}.

\bibitem{Kuperberg:notdual}
Mikhail Khovanov and Greg Kuperberg, \emph{Web bases for $sl(3)$ are not dual
  canonical}, Pacific J. Math. \textbf{188} (1999), no.~1, 129--153,
  \mbox{arXiv:q-alg/9712046}.

\bibitem{Kim:thesis}
Dongseok Kim, \emph{Graphical calculus on representations of quantum {Lie}
  algebras}, Ph.d. thesis, University of California, Davis, 2003,
  \mbox{arXiv:math/0310143}.

\bibitem{Kumar:kacmoody}
Shrawan Kumar, \emph{Kac-moody groups, their flag varieties, and representation
  theory}, Birkh\"auser, 2002.

\bibitem{Kuperberg:spiders}
Greg Kuperberg, \emph{Spiders for rank 2 {Lie} algebras}, Comm. Math. Phys.
  \textbf{180} (1996), no.~1, 109--151, \mbox{arXiv:q-alg/9712003}.

\bibitem{Littelmann:paths}
Peter Littelmann, \emph{Paths and root operators in representation theory},
  Ann. of Math. (2) \textbf{142} (1995), no.~3, 499--525.

\bibitem{Lusztig:qanalog}
George Lusztig, \emph{Singularities, character formulas, and a $q$-analog of
  weight multiplicities}, Analysis and topology on singular spaces, {II}, {III}
  ({Luminy}, 1981), Ast\'erisque, vol. 101, Soc. Math. France, 1983,
  pp.~208--229.

\bibitem{Lusztig:constructible}
\bysame, \emph{Constructible functions on varieties attached to quivers},
  Studies in memory of {Issai Schur} ({Chevaleret/Rehovot}, 2000), Progr.
  Math., vol. 210, Birkh\"auser Boston, 2003, pp.~177--223.

\bibitem{MacPherson:chern}
Robert MacPherson, \emph{Chern classes for singular algebraic varieties}, Ann.
  of Math. (2) \textbf{100} (1974), 423--432.

\bibitem{MV:geometric}
Ivan Mirkovi{\'c} and Kari Vilonen, \emph{Geometric {Langlands} duality and
  representations of algebraic groups over commutative rings}, Ann. of Math.
  (2) \textbf{166} (2007), no.~1, 95--143, \mbox{arXiv:math/0401222}.

\bibitem{MV:bdg}
Ivan Mirkovi\'{c} and Maxim Vybornov, \emph{Quiver varieties and
  {Beilinson-Drinfeld Grassmannians} of type {A}}, \mbox{arXiv:0712.4160}.

\bibitem{Morrison:diagram}
Scott Morrison, \emph{A diagrammatic category for the representation theory of
  {$U_q(sl_n)$}}, Ph.d. thesis, University of California, Berkeley, 2007,
  \mbox{arXiv:0704.1503}.

\bibitem{Nakajima:kacmoody}
Hiraku Nakajima, \emph{Instantons on {ALE} spaces, quiver varieties, and
  {Kac-Moody} algebras}, Duke Math. J. \textbf{76} (1994), no.~2, 365--416.

\bibitem{Penrose:negative}
Roger Penrose, \emph{Applications of negative dimensional tensors},
  Combinatorial Mathematics and Its Applications, Academic Press, 1971,
  pp.~221--244.

\bibitem{RT:ribbon}
Nicolai~Yu. Reshetikhin and Vladimir~G. Turaev, \emph{Ribbon graphs and their
  invariants derived from quantum groups}, Comm. Math. Phys. \textbf{127}
  (1990), no.~1, 1--26.

\bibitem{Ronan:buildings}
Mark Ronan, \emph{Lectures on buildings}, University of Chicago Press, 2009,
  Updated and revised.

\bibitem{Selinger:survey}
Peter Selinger, \emph{A survey of graphical languages for monoidal categories},
  New structures for physics, Lecture Notes in Phys., vol. 813, Springer, 2011,
  \mbox{arXiv:0908.3347}, pp.~289--355.

\bibitem{SK:confluence}
Adam~S. Sikora and Bruce~W. Westbury, \emph{Confluence theory for graphs},
  Algebr. Geom. Topol. \textbf{7} (2007), 439--478, \mbox{arXiv:math/0609832}.

\bibitem{Stroppel:springer}
Catharina Stroppel, \emph{Parabolic category $\mathcal{O}$, perverse sheaves on
  {Grassmannians}, {Springer} fibres and {Khovanov} homology}, Compos. Math.
  \textbf{145} (2009), no.~4, 954--992, \mbox{arXiv:math/0608234}.

\bibitem{Turaev:quantum}
Vladimir~G. Turaev, \emph{Quantum invariants of knots and 3-manifolds}, W. de
  Gruyter, 1994.

\bibitem{Tymoczko:bijection}
Julianna Tymoczko, \emph{A simple bijection between standard $(n,n,n)$ tableaux
  and irreducible webs for $sl_3$}, \mbox{arXiv:1005.4724}.

\bibitem{Westbury:general}
Bruce~W. Westbury, \emph{Web bases for the general linear groups},
  \mbox{arXiv:1011.6542}.

\bibitem{Westbury:trivalent}
\bysame, \emph{Enumeration of non-positive planar trivalent graphs}, J.
  Algebraic Combin. \textbf{25} (2007), no.~4, 357--373,
  \mbox{arXiv:math/0507112}.

\bibitem{Westbury:so7}
\bysame, \emph{Invariant tensors for the spin representation of
  {$\mathrm{so}(7)$}}, Math. Proc. Cambridge Philos. Soc. \textbf{144} (2008),
  no.~1, 217--240, \mbox{arXiv:math/0601209}.

\end{thebibliography}

\providecommand{\bysame}{\leavevmode\hbox to3em{\hrulefill}\thinspace}
\providecommand{\MR}{\relax\ifhmode\unskip\space\fi MR }
\providecommand{\MRhref}[2]{%
  \href{http://www.ams.org/mathscinet-getitem?mr=#1}{#2}
}
\providecommand{\href}[2]{#2}

\end{document}